\documentclass[12pt]{article}
\usepackage{preprint-layout} %
\usepackage{preprint-notation} %
\usepackage{subcaption}

\newcommand{\dist}{\mathrm{dist}}

\makeatletter
\newcommand{\proofparagraph}[1]{%
  \ifnum\prevgraf=0\relax
    {\bfseries\boldmath{#1}}\hspace{0.6em}%
  \else
    \par\addvspace{.5\baselineskip}%
    \noindent{\bfseries\boldmath{#1}}\hspace{0.6em}%
  \fi
}
\makeatother

\title{Error Estimates for Nitsche's Method on Approximate Domains}
\date{\today}

\author{
Mats G. Larson,\quad
Karl Larsson,\quad
Shantiram Mahata
}
\date{}

\begin{document}

\maketitle

\begin{abstract}
We derive a priori error estimates for Nitsche's method applied to elliptic problems on approximate domains. Such approximations arise, for example, in unfitted finite element methods, data-driven simulations, and evolving domain problems, where the computational domain does not coincide exactly with the physical one.

We quantify geometric errors in terms of boundary location and normal perturbations and carry out the analysis in an abstract CutFEM framework under standard stability assumptions. In the energy norm, we obtain an estimate exhibiting an $h^{-1/2}$ amplification of the boundary location error. We then prove a refined $H^1$-seminorm estimate that removes this amplification, yielding a sharper bound with additive contributions from boundary location and normal errors. Finally, we establish an optimal order $L^2$-error estimate based on a refined duality argument, where the geometry contribution appears as a separate additive term, decoupled from the mesh size $h$.

The results reveal a fundamental distinction between the norms: the energy norm amplifies boundary location errors while remaining insensitive to normal perturbations, the $H^1$-seminorm separates location and normal errors, and the $L^2$-norm is insensitive to normal perturbations. This provides a clear characterization of how geometric approximation affects convergence in Nitsche-based finite element methods, with particular relevance for unfitted discretizations.
\end{abstract}

\section{Introduction}
\label{sec:introduction}

The weak imposition of Dirichlet boundary conditions by Nitsche's method~\cite{MR341903} is a standard tool in finite element analysis. While originally developed for fitted meshes, Nitsche's method plays a central role in unfitted or embedded discretizations, where the computational mesh does not conform to the physical boundary. In such settings, the boundary is typically approximated by a surrogate geometry. Approximate geometries also arise in the discretization of CAD models, surface triangulations, data-driven geometries obtained from measurements, and in simulations involving evolving or implicitly defined domains. Geometry approximation leads to a variational crime in the sense of Strang~\cite{MR305547}, where the discrete bilinear and linear forms are assembled on an approximate domain~\cite{MR421108,MR842644,MR2915563}. A key question is how such geometric perturbations affect the accuracy of the numerical solution, and in particular whether their impact differs across commonly used norms. In the classical fitted setting, Bramble, Dupont, and Thom\'ee~\cite{MR343657} showed that boundary approximation errors in a Nitsche formulation can be compensated by boundary value corrections. In modern unfitted methods, in particular within the CutFEM framework~\cite{MR3416285,BuHaLaZa25}, related ideas have been successfully employed; see, e.g.,~\cite{MR3739212}, where optimal convergence is recovered by combining Nitsche's method with a Taylor expansion of the boundary data in the normal direction.
Isoparametric finite elements have also been analysed for unfitted Nitsche based discretizations \cite{MR3829164, MR3975696}, yielding optimal convergence rates.

Existing analyses typically focus on recovering optimal convergence through boundary value corrections or higher-order geometry approximations. In contrast, the present work keeps the geometry errors explicit throughout the estimates and quantifies their different effects on the energy norm, $H^1$-seminorm, and $L^2$-norm.

The analysis in this work is carried out under the assumption that both the exact boundary $\partial\Omega$ and the approximate boundary $\partial\Omega_\delta$ are smooth. While this setting covers a broad class of geometric perturbations and approximations, it excludes polygonal and piecewise smooth geometries arising, for example, from piecewise linear level-set representations commonly used in unfitted finite element methods. Such geometries are nevertheless considered in the numerical experiments to investigate the practical behavior of the method beyond the scope of the present theory.

\paragraph{Contributions.}
We study Nitsche's method posed on a perturbed domain $\Omega_\delta$ and quantify how geometric errors affect the approximation. The analysis is carried out in an abstract CutFEM setting under standard stability and consistency assumptions. Our main contributions are as follows:
\begin{itemize}
\item We prove an energy norm estimate
\begin{equation} \label{eq:intro-energy-norm-error}
\tn u - u_h \tn_h \lesssim h^p + h^{-1/2}\delta
\end{equation}
where $\delta$ measures the boundary location error. This estimate separates discretization and geometry errors, and optimal energy convergence is obtained if $\delta \lesssim h^{p+1/2}$.

\item We derive a refined $H^1$-seminorm estimate
\begin{equation} \label{eq:intro-h1-norm-error}
\|\nabla (u - u_h)\|_{\Omega_\delta} \lesssim h^p + \delta + \delta_n
\end{equation}
where $\delta_n$ measures the boundary normal error. This result removes the $h^{-1/2}$ amplification present in~\eqref{eq:intro-energy-norm-error}, yielding a sharper characterization of geometric effects.

\item We further establish an $L^2$-error estimate
\begin{equation}
\|u - u_h\|_{\Omega_\delta} \lesssim h^{p + 1} + \delta
\end{equation}
The analysis is based on a refined duality argument that exploits the structure of the consistency error and the fact that the dual solution vanishes on $\partial \Omega_\delta$, allowing a sharper and more transparent treatment of the geometry contribution.
\end{itemize}
These results provide a clear separation between bulk discretization errors and boundary-induced geometric errors, and reveal a fundamental distinction between the norms: the energy norm amplifies boundary location errors while remaining insensitive to normal perturbations, the $H^1$-seminorm separates boundary location and normal errors, and the $L^2$-norm decouples the geometry error from the mesh size and is likewise insensitive to normal perturbations. This highlights that geometric perturbations affect different norms in qualitatively different ways.

\paragraph{Outline.}
The remainder of the paper is organized as follows. In Section~\ref{sec:problem-method}, we introduce the model problem, the perturbed domain setting, and the cut finite element formulation, including the abstract assumptions on stabilization. In Section~\ref{sec:standard-error-estimates}, we derive the energy norm estimate and discuss the role of geometric perturbations. Section~\ref{sec:improved-estimate} is devoted to refined error estimates, including both the $H^1$-seminorm estimate and the $L^2$-error estimate based on a refined duality argument. Numerical experiments are given in Section~\ref{sec:numer-examples}, followed by concluding remarks in Section~\ref{sec:conclusions}.

\section{The Problem and Method}
\label{sec:problem-method}

\paragraph{The Problem.}
Consider the Dirichlet boundary value problem
\begin{equation}
-\Delta u = f\quad \text{in $\Omega$}, \qquad u = 0 \quad \text{on $\partial \Omega$}
\label{eq:bvp}
\end{equation}
where $\Omega \subset \IR^d$ is a domain with smooth boundary $\partial \Omega$ and outward unit normal $n$, and $f \in H^{-1}(\Omega)$.
In many situations, the exact domain $\Omega$ is not available in the numerical discretization due to geometric approximation errors, uncertainty in the domain geometry, or because the geometry is only given implicitly. Instead, only an approximate domain $\Omega_\delta$ is available.
Below, we consider the situation where $\Omega_\delta$ is used in Nitsche's method and quantify how the resulting boundary location and normal errors affect the accuracy of the numerical solution.

\paragraph{Approximate Domain.}

Assuming that the approximate boundary $\partial \Omega_\delta$ is smooth with outward unit normal $n_\delta$, we quantify the geometric approximation error through the boundary location and normal errors
\begin{equation}
\| q - \mathrm{id} \|_{L^\infty (\partial \Omega_\delta)} \le \delta
,\qquad
\| n\circ q -  n_\delta \|_{L^\infty (\partial \Omega_\delta)} \le \delta_n
\label{eq:boundary-errors}
\end{equation}
where $q: U_{\delta_0}(\partial \Omega) \to \partial \Omega$ is the closest point mapping onto $\partial \Omega$ defined by
\begin{align}
q(x) = \arg \min_{y \in \partial \Omega} | x - y |
\label{eq:closest-point}
\end{align}
We use the notation of a tubular neighborhood; for a set $\Gamma \subset \IR^d$ and $r > 0$,
\begin{equation}
U_r(\Gamma) := \{ x \in \IR^d : \dist(x,\Gamma) < r \}
\end{equation}
and $\delta_0>0$ is a fixed parameter independent of $\delta$ and $h$ such that the closest point mapping \eqref{eq:closest-point} is well defined and has uniformly bounded Jacobian in $U_{\delta_0}(\partial \Omega)$.
We assume that $\delta \le \delta_0$, so that the closest point mapping $q$ is well defined on $\partial \Omega_\delta$. 
Then the boundary location error in \eqref{eq:boundary-errors} implies that
\(
\partial \Omega_\delta \subset U_{\delta}(\partial \Omega)
\).
We further assume that the domain mismatch is confined to a tubular neighborhood of the two boundaries, namely
\begin{equation}
\Omega_\delta\triangle\Omega
\subset
U_{c\delta}(\partial\Omega)
\cap
U_{c\delta}(\partial\Omega_\delta)
\label{eq:domain-mismatch-assumption}
\end{equation}
with a constant \(c\) independent of \(h\) and \(\delta\). Since \(c\) is fixed, we suppress it in the notation below and write \(U_\delta\) for \(U_{c\delta}\).

For sufficiently small $\delta$, we assume that the family of perturbed domains
$\{\Omega_\delta\}$ remains uniformly regular. In particular, the boundaries
$\partial\Omega_\delta$ are smooth, admit tubular neighborhoods of radius bounded
below independently of $\delta$, and the associated closest-point mappings have
uniformly bounded Jacobians. Furthermore, the trace, extension, and elliptic regularity estimates used below hold with constants independent of $\delta$.

\paragraph{Approximation Space.}
Let $\Omega_0$ be a convex polygonal domain in $\IR^d$ such that $\Omega \cup U_{\delta_0}(\partial \Omega) \subset \Omega_0$. Let $\mcT_{0,h}$ be a quasiuniform mesh of $\Omega_0$ consisting of shape regular elements $T$ with mesh parameter $h\in (0,h_0]$, and let $V_{0,h}\subset H^1(\Omega_0)$ be the associated finite element space of piecewise polynomials of degree $p$. Define the active mesh
\begin{align}
\mcT_h := \{ T\in \mcT_{0,h}: T\cap \Omega_\delta \ne \emptyset \}
\end{align}
and the corresponding active domain $\Omega_h = \bigcup_{T\in \mcT_h} T$.
The discrete space $V_h$ is defined by restriction from $V_{0,h}$,
\begin{align}
    V_h = \{ \hat{v}|_{\Omega_h} : \hat{v} \in V_{0,h} \}
\end{align}

Let $E:H^s(\Omega) \to H^s(\IR^d)$ be a continuous extension operator satisfying
\begin{align}\label{eq:ext-stab}
\| Ev \|_{H^s(\IR^d)} = \| v^e \|_{H^s(\IR^d)} \lesssim \| v \|_{H^s(\Omega)}, \qquad v \in H^s(\Omega)
\end{align}
and write $v^e = Ev$, see \cite{MR290095} for details.
Similarly, we assume that the family of perturbed domains
\(\{\Omega_\delta\}\) admits uniformly bounded extension operators.
More precisely, for each sufficiently small \(\delta\), let
\(
  E_\delta:H^s(\Omega_\delta)\to H^s(\mathbb R^d)
\)
be a continuous extension operator such that
\begin{equation}\label{eq:ext-stab-delta}
  \|E_\delta v\|_{H^s(\mathbb R^d)}
  \lesssim
  \|v\|_{H^s(\Omega_\delta)},
  \qquad v\in H^s(\Omega_\delta)
\end{equation}
where the hidden constant is assumed to be independent of \(\delta\). For
functions defined on \(\Omega_\delta\), we again write \(v^e=E_\delta v\);
the meaning of the extension is determined by the domain on which \(v\)
is originally defined. The bounds \eqref{eq:ext-stab} and \eqref{eq:ext-stab-delta} are assumed to hold for all Sobolev indices \(s\) used in the analysis; in particular, it is sufficient that they hold for $0\le s \le \max(p+1,2+\epsilon)$ with $0<\epsilon<1/2$.

Let $\pi_h: H^s(\Omega_h) \to V_h$ be a Scott--Zhang interpolation operator \cite{MR1011446} which, in its classical form, is a projection onto $V_h$, i.e., $\pi_h v_h = v_h$ for all $v_h \in V_h$.
It satisfies the elementwise interpolation error estimate
\begin{equation}
\|v-\pi_h v\|_{H^m(T)}
\lesssim
h^{s-m}\|v\|_{H^s(N_h(T))},
\qquad
0\le m\le s\le p+1,
\quad m=0,1
\label{eq:interpolation-elm}
\end{equation}
where $N_h(T)$ denotes the element patch consisting of $T$ and its neighboring elements. This operator is applied to functions in $H^s(\Omega)$ by first extending them to $\IR^d$ using the extension operator $E$ and then restricting to $\Omega_h$.

\paragraph{CutFEM on the Approximate Domain.}
The cut finite element method on $\Omega_\delta$ reads: find $u_h\in V_h$ such that
\begin{align}\label{eq:fem}
    A_h(u_h,v) = l_h(v)   \qquad \forall v\in V_h
\end{align}
where the forms are defined by 
\begin{align}
A_h(v,w) &= a_h(v,w) + s_h(v,w)
\label{eq:forms-Ah}
\\
a_h(v,w) &= (\nabla v, \nabla w)_{\Omega_\delta}  - (\nabla_{n_\delta} v, w)_{\partial \Omega_\delta} -  (v, \nabla_{n_\delta} w)_{\partial \Omega_\delta} + \beta h^{-1} (v,w)_{\partial \Omega_\delta}
\label{eq:forms}
\\
l_h(v) &= (f_\delta,v)_{\Omega_\delta}
\label{eq:forms-l}
\end{align}
with \(f_\delta\) denoting an extension of \(f\) from \(\Omega\) to
\(\Omega\cup U_\delta(\partial\Omega)\), satisfying \(f_\delta=f\) in \(\Omega\)
and such that the residual
\(
\Delta u^e+f_\delta
\)
belongs to \(H^s(U_\delta(\partial\Omega))\) and satisfies
\begin{equation}
\|\Delta u^e+f_\delta\|_{H^s(U_\delta(\partial\Omega))}
\lesssim
\|u^e\|_{H^{s+2}(U_\delta(\partial\Omega))},
\qquad
s=\max(p-1,\epsilon)
\label{eq:residual-extension}
\end{equation}

The form $s_h$ in \eqref{eq:forms-Ah} is a stabilization term added to ensure stability of the method regardless of the cut situation, and we detail this form next.

\paragraph{Stabilization Form.}
The stabilization form $s_h$ is a symmetric positive semi-definite bilinear form on $V_h$, and thus induces the semi-norm $\|v\|_{s_h} = s_h(v,v)^{1/2}$. We assume that it satisfies
\begin{align}
\| \nabla v \|_{\Omega_h}^2 &\lesssim  \| \nabla v \|_{\Omega_\delta}^2 + \| v \|_{s_h}^2, \quad v \in V_h
\end{align}
and the weak consistency property
\begin{align}
\| \pi_h w \|_{s_h} \lesssim h^{s-1} \| w \|_{H^{s}(\Omega_h)}, \quad w \in H^{s}(\Omega_h), \quad 1 \le s \le p+1
\label{eq:weak-consistency}
\end{align}
where we in particular note that $\| \pi_h v^e \|_{s_h} \lesssim h^{s-1} \| v \|_{H^{s}(\Omega)}$ for $v \in H^s(\Omega)$.
For the analysis, we extend the stabilization term to non-discrete arguments
defined on \(\Omega_h\) by setting
\begin{equation}
s_h(v,w):=s_h(\pi_hv,\pi_hw),
\qquad v,w\in H^s(\Omega_h)
\end{equation}
For functions originally defined on \(\Omega\), we first apply the extension
operator and then restrict to \(\Omega_h\).
This definition coincides with the original stabilization term for $v,w \in V_h$, since $\pi_h$ is a projection onto $V_h$.

The most common choice of stabilization form $s_h$ satisfying these properties, which we also use in our numerical experiments, is the so-called ghost penalty stabilization \cite{Bu10} defined by
\begin{align}
    s_{h}(v,w) = \sum_{j = 1}^{p} \sum_{F\in \mcF_h} \gamma_j h^{2j-1}\Bigl(\ldb \nabla_{n_{F}}^j v\rdb_F, \ldb \nabla_{n_{F}}^j w\rdb_F\Bigr)_F 
\label{eq:ghost-penalty}
\end{align}
where $\gamma_j>0$,  $\mcF_h$ is the set of interior faces in $\mcT_h(\partial\Omega_\delta) = \{T\in \mcT_h : T\cap \partial\Omega_\delta \ne \emptyset \}$, $n_{F}$ is a fixed unit normal to a face $F\in \mcF_h$,  $\nabla_{n_{F}}^j$ denotes the $j$-th order derivative in the direction of the normal $n_{F}$, and $\ldb\cdot\rdb_F$ stands for the jump operator across the face.

\paragraph{Properties of the Method.}
Let the energy norm be defined by
\begin{align}
\tn v \tn_h^2
= \| \nabla v \|^2_{\Omega_\delta} + h \| \nabla_{n_\delta} v \|^2_{\partial \Omega_\delta} + h^{-1} \| v \|^2_{\partial \Omega_\delta} + \| v \|^2_{s_{h}}
\label{eq:energy-norm}
\end{align}
With respect to this norm, the bilinear form is coercive,
\begin{align}
\tn v \tn_h^2 \lesssim A_{h}(v,v), \qquad v \in V_h
\end{align}
and for $\beta >0$ sufficiently large, continuous,
\begin{align}
A_h(v,w) \lesssim \tn v \tn_h \tn w \tn_h, \qquad v, w \in H^{3/2+\epsilon}(\Omega)+V_h,
\quad \epsilon>0
\end{align}
Here functions originally defined on \(\Omega\) are identified with their
extensions by \(E\), while functions originally defined on \(\Omega_\delta\)
are identified with their extensions by \(E_\delta\).

The following discrete Poincar\'e estimate holds,
\begin{equation}
\| v \|_{\Omega_\delta} \lesssim \tn v \tn_h, \qquad v \in V_h
\label{eq:discrete-poincare}
\end{equation}
which in particular implies that \(\tn \cdot \tn_h\) defines a norm on \(V_h\).

Using coercivity, the discrete Poincar\'e estimate, and the Lax--Milgram lemma, we conclude that there exists a unique solution \(u_h \in V_h\) to \eqref{eq:fem} satisfying the stability estimate
\begin{equation}
\tn u_h \tn_h \lesssim \| f_\delta \|_{H^{-1}(\Omega_\delta)}
\label{eq:stability}
\end{equation}

\section{The Standard Error Estimates}
\label{sec:standard-error-estimates}

\paragraph{Interpolation Estimate.} Using trace inequalities and the interpolation error estimate \eqref{eq:interpolation-elm} we have the following interpolation estimate for the energy norm,
\begin{equation}
\tn v^e - \pi_h v^e  \tn_h \lesssim h^{s-1} \| v \|_{H^{s}(\Omega)}, \quad 3/2 < s\le p+1
\label{eq:interpol-energy}
\end{equation}

\begin{lem}[Normal Slicing Estimate] \label{lem:normal-slicing-estimate}
Let \(Q=\omega\times I\), where \(\omega\subset\mathbb R^{d-1}\) is a bounded
Lipschitz domain and \(I\subset\mathbb R\) is an interval. For \(0<s<1\),
there holds
\begin{equation}
\int_\omega
\|v(x',\cdot)\|_{H^s(I)}^2
\,dx'
\lesssim
\|v\|_{H^s(Q)}^2
\end{equation}
for all \(v\in H^s(Q)\).
\end{lem}

\begin{proof}
Since \(Q\) is a Lipschitz product domain, there exists a bounded extension
operator from \(H^s(Q)\) to \(H^s(\mathbb R^d)\). It is therefore enough to prove the
estimate for \(v\in H^s(\mathbb R^d)\). Taking the Fourier transform in
\((x',t)\), we have
\begin{align}
\int_{\mathbb R^{d-1}}
\|v(x',\cdot)\|_{H^s(\mathbb R)}^2
\,dx'
&=
\int_{\mathbb R^{d-1}}\int_{\mathbb R}
(1+|\tau|^2)^s
|\widehat v(\xi',\tau)|^2
\,d\tau\,d\xi'
\\
&\le
\int_{\mathbb R^{d-1}}\int_{\mathbb R}
(1+|\xi'|^2+|\tau|^2)^s
|\widehat v(\xi',\tau)|^2
\,d\tau\,d\xi'
\\
&=
\|v\|_{H^s(\mathbb R^d)}^2
\end{align}
Restricting the one-dimensional slices from \(\mathbb R\) to \(I\) and the
tangential variable from \(\mathbb R^{d-1}\) to \(\omega\), and using the
boundedness of the extension operator, gives the result.
\end{proof}

\begin{lem}[Tubular Neighborhood Scaling Estimate]
\label{lem:scaling-lemma}
Assume that $\partial \Omega$ is smooth and that $\delta>0$ is sufficiently
small so that the tubular neighborhood $U_\delta(\partial \Omega)$ admits
normal coordinates. Let \(v \in H^s(U_\delta(\partial \Omega))\), with
\(v=0\) in \(U_\delta(\partial \Omega)\cap \Omega\). Assume that either
\(s=m\in \mathbb N\), \(m\ge 1\), or \(0<s<1/2\). Then
\begin{equation}
\|v\|_{L^2(U_\delta(\partial \Omega))}
\lesssim
\delta^s \|v\|_{H^s(U_\delta(\partial \Omega))}
\label{eq:scaling-lemma}
\end{equation}
where the constant is independent of \(\delta\).
\end{lem}

\begin{rem}
In the applications below, the lemma is used with
\(s=\max(p-1,\epsilon)\). For \(p=1\), we choose
$s=\epsilon$ with \(0<\epsilon<1/2\), while for \(p\ge 2\), we take
\(s=p-1\in\mathbb N\). Thus the lemma covers all cases used in the
error estimates.
\end{rem}

\begin{proof}
We first consider the one-dimensional estimate to illustrate the argument, then extend to higher dimensions.

\proofparagraph{One-Dimensional Estimate.}
Let \(I=(-\delta,\delta)\) and let \(v\in H^s(I)\) satisfy
\(v=0\) on \((-\delta,0)\). Hence, it suffices
to estimate \(\|v\|_{L^2(0,\delta)}\). 

\proofparagraph{\(s=m\in\mathbb N\).}
Since \(v=0\) on \((-\delta,0)\) and \(v\in H^m(I)\), the traces of the
derivatives up to order \(m-1\) vanish at the interface,
\begin{equation}
v^{(j)}(0^+)=0,
\qquad j=0,\ldots,m-1
\end{equation}
Thus, for \(0<x<\delta\), Taylor's formula with integral remainder gives
\begin{equation}
v(x)
=
\frac{1}{(m-1)!}
\int_0^x (x-t)^{m-1} v^{(m)}(t)\,dt
\end{equation}
Using Cauchy--Schwarz, we obtain
\begin{align}
|v(x)|^2
&\lesssim
\left(\int_0^x (x-t)^{2m-2}\,dt\right)
\left(\int_0^x |v^{(m)}(t)|^2\,dt\right)
\lesssim
x^{2m-1}\|v^{(m)}\|_{L^2(0,\delta)}^2
\end{align}
and therefore
\begin{align}
\|v\|_{L^2(0,\delta)}^2
&\lesssim
\int_0^\delta x^{2m-1}\,dx
\,
\|v^{(m)}\|_{L^2(0,\delta)}^2
\lesssim
\delta^{2m}\|v\|_{H^m(I)}^2
\end{align}
Hence
\begin{equation}
\|v\|_{L^2(0,\delta)}
\lesssim
\delta^m \|v\|_{H^m(I)}
\end{equation}

\proofparagraph{\(0<s<1/2\).}
By the definition of the fractional seminorm and since \(v=0\) on
\((-\delta,0)\),
\begin{equation}
|v|_{H^s(I)}^2
\gtrsim
\int_0^\delta\int_{-\delta}^0
\frac{|v(x)|^2}{|x-y|^{1+2s}}\,dy\,dx
\end{equation}
For \(0<x<\delta\),
\begin{align}
\int_{-\delta}^0 \frac{dy}{|x-y|^{1+2s}}
&=
\frac{1}{2s}\left(x^{-2s}-(x+\delta)^{-2s}\right)
\gtrsim
\delta^{-2s}
\end{align}
Consequently,
\begin{align}
|v|_{H^s(I)}^2
&\gtrsim
\delta^{-2s}
\int_0^\delta |v(x)|^2\,dx
=
\delta^{-2s}\|v\|_{L^2(0,\delta)}^2
\end{align}
and thus
\begin{equation}
\|v\|_{L^2(0,\delta)}
\lesssim
\delta^s |v|_{H^s(I)}
\le
\delta^s \|v\|_{H^s(I)}
\end{equation}

\proofparagraph{Higher-Dimensional Estimate.}
We cover \(\partial\Omega\) by finitely many coordinate patches and flatten the
boundary. Since the tubular coordinate maps are uniformly regular for
sufficiently small \(\delta\), it is enough to prove the estimate on a
reference patch \(Q_\delta=\omega\times(-\delta,\delta)\), where
\(v=0\) for \(t<0\).

\proofparagraph{\(s=m\in\mathbb N\).} The one-dimensional estimate applied in the normal
variable gives, for smooth \(v\),
\begin{equation}
\int_\omega\int_0^\delta |v(x',t)|^2\,dt\,dx'
\lesssim
\delta^{2m}
\int_\omega\int_0^\delta |\partial_t^m v(x',t)|^2\,dt\,dx'
\end{equation}
and hence
\begin{equation}
\|v\|_{L^2(Q_\delta)}
\lesssim
\delta^m \|v\|_{H^m(Q_\delta)}
\end{equation}
The estimate for general \(v\in H^m(Q_\delta)\) follows by density.

\proofparagraph{\(0<s<1/2\).}
By the normal slicing estimate Lemma~\ref{lem:normal-slicing-estimate},
\begin{equation}
\int_\omega
\|v(x',\cdot)\|_{H^s(-\delta,\delta)}^2
\,dx'
\lesssim
\|v\|_{H^s(Q_\delta)}^2
\end{equation}
Applying the one-dimensional fractional estimate for almost every \(x'\), and
using that \(v=0\) for \(t<0\), gives
\begin{align}
\int_\omega\int_0^\delta |v(x',t)|^2\,dt\,dx'
&\lesssim
\delta^{2s}
\int_\omega
\|v(x',\cdot)\|_{H^s(-\delta,\delta)}^2
\,dx'
\lesssim
\delta^{2s}\|v\|_{H^s(Q_\delta)}^2
\end{align}

\end{proof}

\begin{rem}
The scaling in Lemma~\ref{lem:scaling-lemma} is sharp. Let
\(\widehat v\in C^\infty(-1,1)\) be nonzero, with
\(\widehat v=0\) on \((-1,0)\), and define
\(v_\delta(x)=\widehat v(x/\delta)\). Then
\begin{equation}
\|v_\delta\|_{L^2(-\delta,\delta)}
\sim
\delta^{1/2}
\end{equation}
Moreover, for the values of \(s\) covered by the lemma,
\begin{equation}
\|v_\delta\|_{H^s(-\delta,\delta)}
\lesssim
\delta^{1/2-s}
\end{equation}
for \(0<\delta\le1\). Hence
\begin{equation}
\frac{\|v_\delta\|_{L^2(-\delta,\delta)}}
{\|v_\delta\|_{H^s(-\delta,\delta)}}
\gtrsim
\delta^s
\end{equation}
Thus the factor \(\delta^s\) cannot in general be improved.
\end{rem}

\begin{lem}[Domain Mismatch Estimate]
\label{lem:strip}
Assume that the domain mismatch condition \eqref{eq:domain-mismatch-assumption} holds,  that $\partial \Omega_\delta$ is smooth, and that $\delta>0$ is sufficiently small so that the tubular neighborhood $U_\delta(\partial \Omega_\delta)$ admits normal coordinates.
Then, for $v \in H^1(U_\delta(\partial \Omega_\delta))$, there holds
\begin{equation}
\| v \|^2_{\Omega_\delta \setminus \Omega}
\lesssim
\delta \| v \|^2_{\partial \Omega_\delta}
+
\delta^2 \| \nabla v \|^2_{\Omega_\delta \setminus \Omega}
\label{eq:strip-estimate}
\end{equation}
\end{lem}

\begin{proof}
Using the normal coordinate mapping %
\begin{equation}
\Phi_\delta:\partial\Omega_\delta \times (-\delta,\delta) \to U_\delta(\partial\Omega_\delta),
\qquad
\Phi_\delta(y,t)=y+t n_\delta(y)
\label{eq:normal-coordinate-map-delta}
\end{equation}
and boundedness of its Jacobian, we have
\begin{equation}
\| v \|^2_{\Omega_\delta \setminus \Omega}
\lesssim
\int_{\partial \Omega_\delta}\int_{a_y}^{b_y}
|v(y+t n_\delta(y))|^2 \, dt \, d\sigma(y)
\label{eq:strip-proof-1}
\end{equation}
where $(a_y,b_y) \subset (-\delta,\delta)$ is the set of $t$ such that $y+t n_\delta(y) \in \Omega_\delta \setminus \Omega$, with $b_y-a_y \lesssim \delta$.

For fixed $y$, define $w_y(t) = v(y+t n_\delta(y))$.
By the fundamental theorem of calculus,
\begin{align}
|w_y(t)|^2
&\lesssim
|w_y(0)|^2
+
|t| \int_{a_y}^{b_y} |w_y'(s)|^2 \, ds
,\qquad
|w_y'(s)|
\le
|\nabla v(y+s n_\delta(y))|
\label{eq:strip-proof-3}
\end{align}
Integrating over $(a_y,b_y)$ and using $b_y-a_y \lesssim \delta$ yields
\begin{equation}
\int_{a_y}^{b_y} |w_y(t)|^2 dt
\lesssim
\delta |v(y)|^2
+
\delta^2 \int_{a_y}^{b_y} |\nabla v(y+s n_\delta(y))|^2 ds
\label{eq:strip-proof-4}
\end{equation}
Integrating over $\partial \Omega_\delta$ concludes the proof.
\qedhere
\end{proof}

\begin{thm}[Energy Error Estimate]
\label{thm:energy-error}
Assume that the geometric approximation assumptions
\eqref{eq:boundary-errors} and \eqref{eq:domain-mismatch-assumption} hold,
that \(\delta \lesssim h\), that \(u \in H^{p+1}(\Omega)\), and that its
extension satisfies
\(
u^e \in W^1_\infty(U_\delta(\partial\Omega))
\cap
H^{2+\epsilon}(U_\delta(\partial\Omega))
\)
for \(0<\epsilon<1/2\). Then there is a constant, independent of \(h\),
\(\delta\), \(\delta_n\), and the cut configuration, such that
\begin{equation}
\tn u^e - u_h \tn_h
\lesssim
h^p \|u\|_{H^{p+1}(\Omega)}
+
h^{r-1} \|u^e\|_{H^r(U_\delta(\partial\Omega))}
+
h^{-1/2}\delta \|u^e\|_{W^1_\infty(U_\delta(\partial\Omega))}
\label{eq:energy-error}
\end{equation}
where \(r = \max(p+1,2+\epsilon)\).
\end{thm}

\begin{rem}[Condition for Optimal Energy Norm Convergence]
The estimate in Theorem~\ref{thm:energy-error} shows that the boundary location error enters with an \(h^{-1/2}\)-amplification. This behavior is intrinsic to the Nitsche formulation and reflects the presence of the boundary penalty term \(h^{-1}\|v\|_{\partial\Omega_\delta}^2\) in the definition of the energy norm \eqref{eq:energy-norm}. 
The estimate therefore guarantees optimal order convergence of order \(h^p\)
in the energy norm provided that
\(\delta \lesssim h^{p+1/2}\).
\end{rem}

\begin{rem}[Regularity Assumptions]
The assumption \(u^e \in W^1_\infty(U_\delta(\partial\Omega))\) is used only to control the boundary consistency term through the estimate \eqref{eq:u-boundary-shift}. In particular, if \(d=2,3\) and \(p\ge 2\), then the Sobolev embedding \(H^{p+1}(\Omega) \hookrightarrow W^1_\infty(\Omega)\), together with stability of the extension operator, implies this regularity.

Furthermore, since \(u \in H^{p+1}(\Omega)\) and the extension operator is stable, we have \(u^e \in H^{p+1}(U_\delta(\partial\Omega))\). Together with the additional assumption \(u^e \in H^{2+\epsilon}(U_\delta(\partial\Omega))\) for \(\epsilon>0\), this yields \(u^e \in H^r(U_\delta(\partial\Omega))\) where \(r=\max(p+1,2+\epsilon)\). The additional \(H^{2+\epsilon}\)-regularity is only needed in the case \(p=1\), in order to apply Lemma~\ref{lem:scaling-lemma} with \(s>0\) in the residual estimate.
\end{rem}

\begin{proof}
Splitting the error by adding and subtracting the interpolant, we get
\begin{equation}
\tn u^e - u_h \tn_h \leq \tn u^e - \pi_h u^e \tn_h + \tn \pi_h u^e - u_h \tn_h
\label{eq:error-split}
\end{equation}
where the first term is estimated by interpolation \eqref{eq:interpol-energy}.
Using the coercivity of $A_h$ on $V_h$, we obtain for the second term
\begin{align}
\tn \pi_h u^e - u_h \tn_h
 &\lesssim \sup_{v \in V_h} \frac{A_h(\pi_h u^e - u_h, v)}{\tn v \tn_h} 
 \\
 & = \sup_{v \in V_h}\Big( \underbrace{\frac{A_h(\pi_h u^e - u^e, v)}{\tn v \tn_h} }_{\lesssim \tn u^e - \pi_h u^e \tn_h}
 +  \underbrace{\frac{A_h(u^e, v) - l_h(v)}{\tn v \tn_h}}_{\bigstar} \Big) 
\label{eq:galerkin-orthog}
\end{align}
To estimate $\bigstar$ we use partial integration and $\Delta u^e + f_\delta=0$ in $\Omega$,
\begin{align}
A_h(u^e, v) - l_h(v) &=  (\nabla u^e, \nabla v)_{\Omega_\delta}  - (\nablandel u^e, v)_{\partial \Omega_\delta} -  (u^e, \nablandel v)_{\partial \Omega_\delta}
\\
&\qquad  + \beta h^{-1} (u^e,v)_{\partial \Omega_\delta} + s_h(\pi_h u^e,v) - (f_\delta,v)_{\Omega_\delta}
\nonumber\\
&= - (\Delta u^e + f_\delta,v)_{\Omega_\delta} -  (u^e, \nablandel v)_{\partial \Omega_\delta}
\label{eq:energy-norm-middle-term}
\\&\qquad \nonumber
+ \beta h^{-1} (u^e,v)_{\partial \Omega_\delta} + s_h(\pi_h u^e,v)
\\
&\lesssim \|\Delta u^e + f_\delta\|_{\Omega_\delta\setminus \Omega} \|v\|_{\Omega_\delta\setminus \Omega} +h^{-1/2}\|u^e\|_{\partial \Omega_\delta} h^{1/2}\| \nablandel v\|_{\partial \Omega_\delta} 
\\
&\qquad + \beta h^{-1/2}\|u^e\|_{\partial \Omega_\delta}  h^{-1/2}\|v\|_{\partial \Omega_\delta} 
+ \| \pi_h u^e \|_{s_h} \|v\|_{s_h}
\nonumber
\end{align}
Since $\Omega_\delta\setminus \Omega \subset U_\delta(\partial \Omega)$, $\Delta u^e + f_\delta = 0$ in $\Omega$, and 
$\Delta u^e + f_\delta \in H^s(U_\delta(\partial \Omega))$ we may employ Lemma \ref{lem:scaling-lemma} and the residual stability \eqref{eq:residual-extension} to get
\begin{align}
\|\Delta u^e + f_\delta\|_{\Omega_\delta\setminus \Omega}
&\le \|\Delta u^e + f_\delta\|_{U_\delta(\partial \Omega)}
\label{eq:delta-s0}
\\&
\lesssim \delta^{s} \|\Delta u^e + f_\delta\|_{H^s(U_\delta(\partial \Omega))}
\\&
\lesssim \delta^{s} \|u^e\|_{H^{s+2}(U_\delta(\partial \Omega))}
\label{eq:delta-s}
\end{align}
Next, using Lemma \ref{lem:strip} we have
\begin{align}
\| v \|^2_{\Omega_\delta \setminus \Omega} &\lesssim  \delta \| v \|^2_{\partial \Omega_\delta} + \delta^2 \|\nabla v \|^2_{\Omega_\delta \setminus \Omega}
\\
&
\lesssim \delta h  h^{-1} \| v \|^2_{\partial \Omega_\delta} + \delta^2 \|\nabla v \|^2_{\Omega_\delta \setminus \Omega} 
\lesssim \delta( h + \delta) \tn v \tn_h^2
\label{eq:strip-estimate-energy}
\end{align}
Taking \(s=\max(p-1,\epsilon)\) in \eqref{eq:delta-s}, and combining this with \eqref{eq:weak-consistency} and \eqref{eq:strip-estimate-energy}, we obtain
\begin{align}
\|\Delta u^e + f_\delta\|_{\Omega_\delta\setminus \Omega} \|v\|_{\Omega_\delta\setminus \Omega}
&\lesssim
\bigl((\delta^{s+1/2} h^{1/2} + \delta^{s+1})\bigr)
\|u^e\|_{H^{s+2}(U_\delta(\partial \Omega))}
\tn v \tn_h
\\
&\lesssim
h^{s+1} \|u^e\|_{H^{s+2}(U_\delta(\partial \Omega))}
\tn v \tn_h
\label{eq:prelim-energy-error}
\end{align}
where, in the last inequality, we used that the geometric error satisfies $\delta \lesssim h$ to recover optimal order convergence.

For the boundary terms, we use the estimate
\begin{equation}
\|u^e\|_{\partial \Omega_\delta} \lesssim \delta \|u^e\|_{W^1_\infty(U_\delta(\partial \Omega))}
\label{eq:u-boundary-shift}
\end{equation}
which follows from the mean value theorem and the fact that $u^e = 0$ on $\partial \Omega$ and $\partial \Omega_\delta$ lies within distance $\delta$ of $\partial \Omega$.
Combining this estimate with \eqref{eq:prelim-energy-error}, we conclude that
\begin{align}
&|A_h(u^e,v) - l_h(v)|
\nonumber
\\ &\qquad
\lesssim
\Big( h^p \|u\|_{H^{p+1}(\Omega)} + h^{r-1} \|u^e\|_{H^{r}(U_\delta(\partial \Omega))} + h^{-1/2}\delta \|u^e\|_{W^1_\infty(U_\delta(\partial \Omega))} \Big) \tn v \tn_h
\label{eq:consistency}
\end{align}
where $r = s+2 = \max(p+1,2+\epsilon)$.
Inserting this bound and the interpolation estimate \eqref{eq:interpol-energy} into \eqref{eq:galerkin-orthog}, we obtain the desired result \eqref{eq:energy-error}.
\end{proof}

\section{Refined Error Estimates}
\label{sec:improved-estimate}
We define the Nitsche normal flux on the boundary $\partial \Omega_\delta$ by
\begin{equation}
\Sigma_{n_\delta}(v) = \nabla_{n_\delta} v - \beta h^{-1} v, \qquad v \in V_h
\label{eq:flux}
\end{equation}
which is the consistent boundary flux associated with the Nitsche formulation.
Using the triangle inequality, we directly conclude that
\begin{equation}
h^{1/2} \| \Sigma_{n_\delta}(v) \|_{\partial \Omega_\delta} \lesssim \tn v \tn_h, \qquad v \in V_h
\label{eq:flux-inverse}
\end{equation}
In the next lemma, we will control the Nitsche normal flux in the $H^{-1}(\partial\Omega_\delta)$-norm, where
\(H^{-1}(\partial\Omega_\delta):=\bigl(H^1(\partial\Omega_\delta)\bigr)'\)
with dual norm
\begin{equation}
\|\xi\|_{H^{-1}(\partial\Omega_\delta)}
:=
\sup_{\mu\in H^1(\partial\Omega_\delta)\setminus\{0\}}
\frac{\langle \xi,\mu\rangle_{\partial\Omega_\delta}}
{\|\mu\|_{H^1(\partial\Omega_\delta)}}
\end{equation}
For \(\xi\in L^2(\partial\Omega_\delta)\), and in particular for
\(\xi=\Sigma_{n_\delta}(v)\), the duality pairing is identified with the
\(L^2(\partial\Omega_\delta)\)-inner product,
\begin{equation}
\langle \xi,\mu\rangle_{\partial\Omega_\delta}
=
(\xi,\mu)_{\partial\Omega_\delta}
\end{equation}

\paragraph{Discrete Dual Problem.}
Given \(\psi\in [L^2(\Omega_\delta)]^d\), find \(\phi_h\in V_h\) such that
\begin{equation}
A_h(v,\phi_h)=(\nabla v,\psi)_{\Omega_\delta}
\qquad \forall v\in V_h
\label{eq:discrete-dual-problem}
\end{equation}
We note that 
with $v = \phi_h$ we obtain the stability estimate 
\begin{align}
\tn \phi_h \tn_h \lesssim \| \psi \|_{\Omega_\delta}
\label{eq:dual-stab}
\end{align}

\begin{lem}[Discrete Boundary Flux Estimate]
\label{lem:boundary-flux-estimate}
Let \(\phi_h\in V_h\) solve the discrete dual problem \eqref{eq:discrete-dual-problem}. Then
\begin{equation}
\|\Sigma_{n_\delta}(\phi_h)\|_{H^{-1}(\partial\Omega_\delta)}
\lesssim
\|\psi\|_{\Omega_\delta}
\label{eq:boundary-flux-estimate}
\end{equation}
where
the constant is independent of \(h\), \(\delta\), and the cut configuration.
\end{lem}

\begin{proof}
Let \(\mu\in H^1(\partial\Omega_\delta)\). We show that
\begin{equation}
|(\mu,\Sigma_{n_\delta}(\phi_h))_{\partial\Omega_\delta}|
\lesssim
\|\mu\|_{H^1(\partial\Omega_\delta)}
\|\psi\|_{\Omega_\delta}
\label{eq:flux-duality-bound}
\end{equation}
By the assumed uniform regularity of the family \(\{\Omega_\delta\}\),
let \(\eta\in H^{3/2}(\Omega_\delta)\) be a lifting of \(\mu\) satisfying
\begin{equation}
  \eta|_{\partial\Omega_\delta}=\mu,
  \qquad
  \|\eta\|_{H^{3/2}(\Omega_\delta)}
  \lesssim
  \|\mu\|_{H^1(\partial\Omega_\delta)}
\end{equation}
where the hidden constant is independent of \(\delta\). Let
\(\eta^e=E_\delta\eta\) denote its extension to \(\mathbb R^d\).
Then, by the uniform stability of \(E_\delta\) \eqref{eq:ext-stab-delta} and restriction to the active domain,
\begin{equation}
\|\eta^e\|_{H^{3/2}(\Omega_h)}
\lesssim
\|\eta\|_{H^{3/2}(\Omega_\delta)}
\lesssim
\|\mu\|_{H^1(\partial\Omega_\delta)}
\label{eq:trace-lifting}
\end{equation}
Since \(\eta^e=\mu\) on \(\partial\Omega_\delta\), we split
\begin{align}
(\mu,\Sigma_{n_\delta}(\phi_h))_{\partial\Omega_\delta}
&=
(\eta^e-\pi_h\eta^e,\Sigma_{n_\delta}(\phi_h))_{\partial\Omega_\delta}
+
(\pi_h\eta^e,\Sigma_{n_\delta}(\phi_h))_{\partial\Omega_\delta}
\label{eq:flux-split}
\end{align}
For the first term, using the cut-boundary interpolation estimate
\begin{equation}
\|\eta^e-\pi_h\eta^e\|_{\partial\Omega_\delta}
\lesssim
h\|\eta^e\|_{H^{3/2}(\Omega_h)}
\label{eq:boundary-interpolation-eta}
\end{equation}
which follows by summing the cut trace inequality 
\begin{equation}
\|v\|_{T\cap\partial\Omega_\delta}^2 \lesssim h^{-1}\|v\|_T^2 + h\|\nabla v\|_T^2
\label{eq:cut-trace-inequality}
\end{equation}
over \(T\in\mathcal T_h(\partial\Omega_\delta)\),
with constants independent of how \(\partial\Omega_\delta\) cuts the elements,
and using the interpolation estimate \eqref{eq:interpolation-elm},
together with
\eqref{eq:flux-inverse} and \eqref{eq:dual-stab}, we obtain
\begin{align}
|(\eta^e-\pi_h\eta^e,\Sigma_{n_\delta}(\phi_h))_{\partial\Omega_\delta}|
&\lesssim
h\|\eta^e\|_{H^{3/2}(\Omega_h)}
\|\Sigma_{n_\delta}(\phi_h)\|_{\partial\Omega_\delta}
\\
&\lesssim
h^{1/2}\|\eta^e\|_{H^{3/2}(\Omega_h)}
\tn\phi_h\tn_h
\\
&\lesssim
\|\eta^e\|_{H^{3/2}(\Omega_h)}
\|\psi\|_{\Omega_\delta}
\label{eq:flux-first-term}
\end{align}
For the second term in \eqref{eq:flux-split}, using the definition of \(A_h\), the definition of
\(\Sigma_{n_\delta}\), and the discrete dual problem, we have
\begin{align}
(\pi_h\eta^e,\Sigma_{n_\delta}(\phi_h))_{\partial\Omega_\delta}
&=
(\nabla\pi_h\eta^e,\nabla\phi_h)_{\Omega_\delta}
-
(\nabla_{n_\delta}\pi_h\eta^e,\phi_h)_{\partial\Omega_\delta}
\\&\qquad\nonumber
+
s_h(\pi_h\eta^e,\phi_h)
-
A_h(\pi_h\eta^e,\phi_h)
\\
&=
(\nabla\pi_h\eta^e,\nabla\phi_h)_{\Omega_\delta}
-
(\nabla_{n_\delta}\pi_h\eta^e,\phi_h)_{\partial\Omega_\delta}
\label{eq:flux-second-identity}
\\&\qquad\nonumber
+
s_h(\pi_h\eta^e,\phi_h)
-
(\nabla\pi_h\eta^e,\psi)_{\Omega_\delta}
\end{align}
Using \(H^1\)-stability of \(\pi_h\), the cut trace inequality \eqref{eq:cut-trace-inequality}, an inverse estimate on \(V_h\), the weak consistency \eqref{eq:weak-consistency}, and the stability \eqref{eq:dual-stab}, we get
\begin{align}
|(\nabla\pi_h\eta^e,\nabla\phi_h)_{\Omega_\delta}|
&\le
\|\nabla\pi_h\eta^e\|_{\Omega_\delta}
\|\nabla\phi_h\|_{\Omega_\delta}
\lesssim
\|\eta^e\|_{H^{3/2}(\Omega_h)}
\|\psi\|_{\Omega_\delta}
\label{eq:flux-grad-term}
\\
|(\nabla_{n_\delta}\pi_h\eta^e,\phi_h)_{\partial\Omega_\delta}|
&\lesssim
h^{1/2}\|\nabla\pi_h\eta^e\|_{\partial\Omega_\delta}
h^{-1/2}\|\phi_h\|_{\partial\Omega_\delta}
\lesssim
\|\eta^e\|_{H^{3/2}(\Omega_h)}
\|\psi\|_{\Omega_\delta}
\label{eq:flux-normal-term}
\\
|s_h(\pi_h\eta^e,\phi_h)|
&\le
\|\pi_h\eta^e\|_{s_h}\|\phi_h\|_{s_h}
\lesssim
h^{1/2}\|\eta^e\|_{H^{3/2}(\Omega_h)}
\tn\phi_h\tn_h
\\&
\lesssim
\|\eta^e\|_{H^{3/2}(\Omega_h)}
\|\psi\|_{\Omega_\delta}
\label{eq:flux-stab-term}
\\
|(\nabla\pi_h\eta^e,\psi)_{\Omega_\delta}|
&\le
\|\nabla\pi_h\eta^e\|_{\Omega_\delta}
\|\psi\|_{\Omega_\delta}
\lesssim
\|\eta^e\|_{H^{3/2}(\Omega_h)}
\|\psi\|_{\Omega_\delta}
\label{eq:flux-rhs-term}
\end{align}
Combining \eqref{eq:flux-second-identity}--\eqref{eq:flux-rhs-term} and
using \eqref{eq:trace-lifting}, we obtain
\begin{equation}
|(\pi_h\eta^e,\Sigma_{n_\delta}(\phi_h))_{\partial\Omega_\delta}|
\lesssim
\|\mu\|_{H^1(\partial\Omega_\delta)}
\|\psi\|_{\Omega_\delta}
\label{eq:flux-second-bound}
\end{equation}
Together with \eqref{eq:flux-first-term}, this proves
\eqref{eq:flux-duality-bound}. Taking the supremum over all
\(\mu\in H^1(\partial\Omega_\delta)\) with
\(\|\mu\|_{H^1(\partial\Omega_\delta)}=1\) gives
\eqref{eq:boundary-flux-estimate}.
\end{proof}

\begin{thm}[Improved $H^1$ Error Estimate]
\label{thm:h1-error}
Assume that the geometric approximation assumptions \eqref{eq:boundary-errors} and \eqref{eq:domain-mismatch-assumption} hold, that \(\delta \lesssim h\), that
\(
u \in H^{p+1}(\Omega)
\),
and that its extension satisfies
\(
u^e \in W^2_\infty(U_\delta(\partial\Omega)) \cap H^{2+\epsilon}(U_\delta(\partial\Omega))
\) for \(0 < \epsilon < 1/2\).
Then there is a constant, independent of \(h\), \(\delta\), \(\delta_n\), and the cut configuration, such that
\begin{equation}
\|\nabla (u^e - u_h) \|_{\Omega_\delta}
\lesssim
(h^p + \delta_n) \| u \|_{H^{p+1}(\Omega)}
+
h^{r-1} \| u^e \|_{H^{r}(U_\delta(\partial\Omega))}
+
\delta \| u^e \|_{W^2_\infty(U_\delta(\partial\Omega))}
\label{eq:h1-error}
\end{equation}
where $r = \max(p+1,2+\epsilon)$.
\end{thm}

\begin{rem}[Condition for Optimal $H^1$-Seminorm Convergence]
The estimate in Theorem~\ref{thm:h1-error} shows that the boundary location error and the normal approximation error enter the \(H^1\)-seminorm additively, without the \(h^{-1/2}\)-amplification present in the energy norm estimate. The estimate therefore guarantees optimal order convergence of order \(h^p\) in the \(H^1\)-seminorm provided that
\(\delta \lesssim h^p\) and \(\delta_n \lesssim h^p\).
\end{rem}

\begin{rem}[Additional Regularity]
Compared to the energy norm estimate, Theorem~\ref{thm:h1-error} additionally assumes
\(u^e \in W^2_\infty(U_\delta(\partial\Omega))\)
in order to control the boundary flux term, in particular the tangential gradient \(\nabla_{\partial\Omega_\delta} u^e\) on \(\partial\Omega_\delta\).

This is a boundary regularity assumption: it is only used to control geometric effects near \(\partial\Omega_\delta\), and does not influence the bulk approximation error or interpolation estimates. A sufficient condition for this assumption is, by Sobolev embedding,
\(
u^e \in H^{2+d/2+\tilde\epsilon}(U_\delta(\partial\Omega))
\)
for some \(\tilde\epsilon>0\).
\end{rem}

\begin{proof}
Splitting the error by adding and subtracting the interpolant, we get
\begin{equation}
\| \nabla (u^e - u_h) \|_{\Omega_\delta}
\le
\| \nabla (u^e - \pi_h u^e) \|_{\Omega_\delta}
+ 
\| \nabla (\pi_h u^e - u_h) \|_{\Omega_\delta}
\label{eq:error-split}
\end{equation}
The first term is estimated by the interpolation estimate and the stability of
extension,
\begin{equation}
\| \nabla (u^e - \pi_h u^e) \|_{\Omega_\delta}
\lesssim
h^p \|u^e\|_{H^{p+1}(\Omega_\delta)}
\lesssim
h^p \|u\|_{H^{p+1}(\Omega)}
\label{eq:interpol-h1}
\end{equation}
It remains to estimate the second term in \eqref{eq:error-split}.
Let \(\psi\in [C_0^\infty(\Omega_\delta)]^d\), and let \(\phi_h\in V_h\)
solve the discrete dual problem \eqref{eq:discrete-dual-problem}.
Setting \(v=\pi_h u^e-u_h\) in the discrete dual problem, we obtain
\begin{align}
(\nabla(\pi_h u^e-u_h),\psi)_{\Omega_\delta}
&=
A_h(\pi_h u^e-u_h,\phi_h)
\\
&=
A_h(\pi_h u^e-u^e,\phi_h)
+
A_h(u^e,\phi_h)-l_h(\phi_h)
\label{eq:error-rep}
\end{align}
The first term is bounded by continuity of \(A_h\), interpolation, and
\eqref{eq:dual-stab},
\begin{align}
|A_h(\pi_h u^e-u^e,\phi_h)|
&\lesssim
\tn \pi_h u^e-u^e \tn_h \tn \phi_h \tn_h
\\
&\lesssim
h^p\|u\|_{H^{p+1}(\Omega)}\|\psi\|_{\Omega_\delta}
\label{eq:interp-continuity-h1}
\end{align}
For the consistency term, integration by parts gives
\begin{align}
A_h(u^e,\phi_h)-l_h(\phi_h)
&=
-\underbrace{(u^e,\Sigma_{n_\delta}(\phi_h))_{\partial\Omega_\delta}}_{I}
-\underbrace{(\Delta u^e+f_\delta,\phi_h)_{\Omega_\delta}}_{II}
+\underbrace{s_h(\pi_hu^e,\phi_h)}_{III}
\label{eq:consistency-dual}
\end{align}

\proofparagraph{Term \(I\).}
Using Lemma~\ref{lem:boundary-flux-estimate}, we obtain
\begin{align}
|I|
&=
|(u^e,\Sigma_{n_\delta}(\phi_h))_{\partial\Omega_\delta}|
\\
&\le
\|u^e\|_{H^1(\partial\Omega_\delta)}
\|\Sigma_{n_\delta}(\phi_h)\|_{H^{-1}(\partial\Omega_\delta)}
\\
&\lesssim
\|u^e\|_{H^1(\partial\Omega_\delta)}
\|\psi\|_{\Omega_\delta}
\label{eq:i-flux-duality}
\end{align}
It remains to estimate
\(\|u^e\|_{H^1(\partial\Omega_\delta)} = \bigl( \|u^e\|_{\partial\Omega_\delta}^2 + \|\nabla_{\partial\Omega_\delta}u^e\|_{\partial\Omega_\delta}^2 \bigr)^{1/2}\). Since \(u^e=0\) on \(\partial\Omega\), a mean value estimate along normals gives
\begin{equation}
\|u^e\|_{\partial\Omega_\delta}
\lesssim
\delta\|u^e\|_{W^1_\infty(U_\delta(\partial\Omega))}
\label{eq:u-boundary-h1-l2}
\end{equation}
see \eqref{eq:u-boundary-shift}.
For the tangential derivative term in \(\|u^e\|_{H^1(\partial\Omega_\delta)}\),
we recall that the tangential gradient on \(\partial\Omega_\delta\) can be expressed as \(\nabla_{\partial\Omega_\delta} = P_{\partial\Omega_\delta}\nabla\), where \(P_{\partial\Omega_\delta} = I - n_\delta\otimes n_\delta\) is a projection onto the tangential plane.
Analogously, the tangential gradient on \(\partial\Omega\) is given by \(\nabla_{\partial\Omega}=P_{\partial\Omega}\nabla\), with
\(P_{\partial\Omega}=I-n\otimes n\).
Here and below, \(n\) and \(P_{\partial\Omega}\) on \(\partial\Omega_\delta\)
denote the pullbacks \(n\circ q\) and \(P_{\partial\Omega}\circ q\).
We can then write
\begin{align}
\|\nabla_{\partial\Omega_\delta}u^e\|_{\partial\Omega_\delta}
&=
\|P_{\partial\Omega_\delta}\nabla u^e\|_{\partial\Omega_\delta}
\\
&\le
\|P_{\partial\Omega_\delta}P_{\partial\Omega}\nabla u^e\|_{\partial\Omega_\delta}
+
\|P_{\partial\Omega_\delta}(I-P_{\partial\Omega})\nabla u^e\|_{\partial\Omega_\delta}
\label{eq:tangential-bound-terms}
\\
&\lesssim
\|P_{\partial\Omega}\nabla u^e\|_{\partial\Omega_\delta}
+
\bigl(
\|n_\delta-n\|_{L^\infty(\partial\Omega_\delta)}
+
\|n_\delta-n\|_{L^\infty(\partial\Omega_\delta)}^2
\bigr)
\|\nabla_n u^e\|_{\partial\Omega_\delta}
\\
&\lesssim
\|\nabla_{\partial\Omega}u^e\|_{\partial\Omega_\delta}
+
\delta_n\|\nabla_n u^e\|_{\partial\Omega_\delta}
\label{eq:tangential-bound-h1}
\end{align}
The estimate of the second term in \eqref{eq:tangential-bound-terms} follows from the identity
\begin{align}
P_{\partial\Omega_\delta}(I-P_{\partial\Omega})
&=
P_{\partial\Omega_\delta}(n\otimes n)
\\
&=
(n\otimes n)
-
(n_\delta\cdot n)(n_\delta\otimes n)
\\
&=
(n-n_\delta)\otimes n
+
(1-n_\delta\cdot n)n_\delta\otimes n
\label{eq:proj-identity-h1}
\end{align}
together with
\begin{equation}
1-n_\delta\cdot n
=
\frac12\|n-n_\delta\|_{\mathbb R^d}^2
\end{equation}
By the normal approximation assumption \eqref{eq:boundary-errors} we arrive at \eqref{eq:tangential-bound-h1}.
Since \(\nabla_{\partial\Omega}u^e=0\) on
\(\partial\Omega\), which directly follows from \(u^e=u=0\) on \(\partial\Omega\), another mean value estimate gives
\begin{equation}
\|\nabla_{\partial\Omega}u^e\|_{\partial\Omega_\delta}
\lesssim
\delta\|u^e\|_{W^2_\infty(U_\delta(\partial\Omega))}
\label{eq:tangential-shift-h1}
\end{equation}
Furthermore, by the uniform trace estimate on \(\partial\Omega_\delta\) and
the stability of the extension,
\begin{equation}
\|\nabla_n u^e\|_{\partial\Omega_\delta}
\lesssim
\|u\|_{H^{p+1}(\Omega)}
\label{eq:normal-trace-h1}
\end{equation}
Combining \eqref{eq:u-boundary-h1-l2}--\eqref{eq:normal-trace-h1}, we get
\begin{equation}
\|u^e\|_{H^1(\partial\Omega_\delta)}
\lesssim
\delta\|u^e\|_{W^2_\infty(U_\delta(\partial\Omega))}
+
\delta_n\|u\|_{H^{p+1}(\Omega)}
\label{eq:u-boundary-h1-final}
\end{equation}
Thus
\begin{equation}
|I|
\lesssim
\Bigl(
\delta\|u^e\|_{W^2_\infty(U_\delta(\partial\Omega))}
+
\delta_n\|u\|_{H^{p+1}(\Omega)}
\Bigr)
\|\psi\|_{\Omega_\delta}
\label{eq:i-est-final}
\end{equation}

\proofparagraph{Term \(II\).}
Similarly to the bound of the corresponding term in the standard energy estimate \eqref{eq:energy-norm-middle-term}, we by
using that \(\Delta u^e+f_\delta=0\) in \(\Omega\), Lemma~\ref{lem:scaling-lemma}, and
\eqref{eq:dual-stab}, obtain
\begin{align}
|II|
&=
|(\Delta u^e+f_\delta,\phi_h)_{\Omega_\delta}|
\\
&\lesssim
\|\Delta u^e+f_\delta\|_{\Omega_\delta\setminus\Omega}
\|\phi_h\|_{\Omega_\delta\setminus\Omega}
\\
&\lesssim
h^{r-1}\|u^e\|_{H^r(U_\delta(\partial\Omega))}
\tn \phi_h \tn_h
\\
&\lesssim
h^{r-1}\|u^e\|_{H^r(U_\delta(\partial\Omega))}
\|\psi\|_{\Omega_\delta}
\label{eq:ii-est-final}
\end{align}
where \(r=\max(p+1,2+\epsilon)\).

\proofparagraph{Term \(III\).}
Using weak consistency of the stabilization \eqref{eq:weak-consistency} and \eqref{eq:dual-stab},
we obtain
\begin{align}
|III|
&=
|s_h(\pi_hu^e,\phi_h)|
\\
&\le
\|\pi_hu^e\|_{s_h}\|\phi_h\|_{s_h}
\\
&\lesssim
h^p\|u\|_{H^{p+1}(\Omega)}\|\psi\|_{\Omega_\delta}
\label{eq:iii-est-final}
\end{align}

\proofparagraph{Conclusion.}
Combining \eqref{eq:error-rep}, \eqref{eq:interp-continuity-h1},
\eqref{eq:i-est-final}, \eqref{eq:ii-est-final}, and
\eqref{eq:iii-est-final}, we find that
\begin{align}
| (\nabla(\pi_h u^e-u_h),\psi)_{\Omega_\delta} |
&\lesssim
\Bigl(
(h^p+\delta_n)\|u\|_{H^{p+1}(\Omega)}
+h^{r-1}\|u^e\|_{H^r(U_\delta(\partial\Omega))}
\label{eq:discrete-error-dual-bound}
\\
&\qquad
+
\delta\|u^e\|_{W^2_\infty(U_\delta(\partial\Omega))}
\Bigr)
\|\psi\|_{\Omega_\delta}
\nonumber
\end{align}
Taking the supremum over all \(\psi\in [C_0^\infty(\Omega_\delta)]^d\) with
\(\|\psi\|_{\Omega_\delta}=1\), and using the density of
\([C_0^\infty(\Omega_\delta)]^d\) in \([L^2(\Omega_\delta)]^d\), gives
\begin{align}
\|\nabla(\pi_h u^e-u_h)\|_{\Omega_\delta}
&\lesssim
(h^p+\delta_n)\|u\|_{H^{p+1}(\Omega)}
+h^{r-1}\|u^e\|_{H^r(U_\delta(\partial\Omega))}
\label{eq:discrete-error-h1-final}
\\
&\qquad
+
\delta\|u^e\|_{W^2_\infty(U_\delta(\partial\Omega))}
\nonumber
\end{align}
Together with \eqref{eq:error-split} and \eqref{eq:interpol-h1}, this proves
\eqref{eq:h1-error}.
\end{proof}

\paragraph{Uniform Elliptic Regularity.}
We assume that the perturbed domains $\Omega_\delta$ satisfy the following uniform elliptic regularity estimate: for each $g\in L^2(\Omega_\delta)$, the solution $\phi\in H^1_0(\Omega_\delta)$ of
\(-\Delta \phi = g\) in \(\Omega_\delta\) satisfies
\begin{equation}
\|\phi\|_{H^2(\Omega_\delta)}
\lesssim
\|g\|_{\Omega_\delta}
\label{eq:uniform-elliptic-regularity}
\end{equation}
with a constant independent of $h$ and $\delta$.

\begin{thm}[$L^2$ Error Estimate]
\label{thm:l2-error}
Assume that the geometric approximation assumptions \eqref{eq:boundary-errors} and \eqref{eq:domain-mismatch-assumption} hold, that \(\delta \lesssim h\), that
\(u \in H^{p+1}(\Omega)\),
that its extension satisfies
\(u^e \in W^1_\infty(U_\delta(\partial\Omega)) \cap H^{2+\epsilon}(U_\delta(\partial\Omega))\) for \(0 < \epsilon < 1/2\), and that the uniform elliptic regularity estimate \eqref{eq:uniform-elliptic-regularity} holds on \(\Omega_\delta\).
Then there is a constant, independent of \(h\), \(\delta\), \(\delta_n\), and the cut configuration, such that
\begin{equation}
\|u^e-u_h\|_{\Omega_\delta}
\lesssim
h^{p+1}
\|u\|_{H^{p+1}(\Omega)}
+
(h^{r}+\delta^{r-1})
\|u^e\|_{H^{r}(U_\delta(\partial\Omega))}
+
\delta
\|u^e\|_{W^1_\infty(U_\delta(\partial\Omega))}
\label{eq:l2-error}
\end{equation}
where $r = \max(p+1,2+\epsilon)$.
\end{thm}

\begin{rem}[Condition for Optimal $L^2$ Convergence]
The estimate in Theorem~\ref{thm:l2-error} holds under the same assumptions as the energy norm estimate. It therefore guarantees optimal order convergence of order \(h^{p+1}\) in the \(L^2\)-norm provided that the boundary location error additionally satisfies
\(\delta \lesssim h^{p+1}\).
\end{rem}

\begin{proof}
Set $e_h = u^e-u_h$
and let \(\phi \in H^1_0(\Omega_\delta)\) solve the dual problem
\begin{equation}
-\Delta \phi = e_h
\qquad \text{in } \Omega_\delta
\label{eq:l2-dual}
\end{equation}
By the uniform elliptic regularity estimate \eqref{eq:uniform-elliptic-regularity}, the dual solution \(\phi\) satisfies
\begin{equation}
\|\phi\|_{H^2(\Omega_\delta)} \lesssim \|e_h\|_{\Omega_\delta}
\label{eq:l2-dual-regularity}
\end{equation}
Using the dual problem, integration by parts, and that $\phi=0$ on $\partial\Omega_\delta$, we can express the square of the $L^2$ error as
\begin{align}
\|e_h\|_{\Omega_\delta}^2
&=
(e_h,e_h)_{\Omega_\delta}
\label{eq:l2-start}
\\
&=
(e_h,-\Delta\phi)_{\Omega_\delta}
\label{eq:l2-start-2}
\\
&=
(\nabla e_h,\nabla \phi)_{\Omega_\delta}
-
(e_h,\nabla_{n_\delta}\phi)_{\partial\Omega_\delta}
\\&=
a_h(e_h,\phi)
\\&=
a_h(e_h,\phi-\pi_h\phi^e)
+
a_h(e_h,\pi_h\phi^e)
\label{eq:l2-parts}
\end{align}
Here \(\phi^e=E_\delta\phi\) denotes the extension of \(\phi\) to \(\mathbb R^d\). By the uniform stability of \(E_\delta\) \eqref{eq:ext-stab-delta}, and by restriction to the active domain, we have \(\|\phi^e\|_{H^2(\Omega_h)}\le \|\phi^e\|_{H^2(\mathbb R^d)} \lesssim\|\phi\|_{H^2(\Omega_\delta)}\).

The first term in \eqref{eq:l2-parts} is estimated using continuity of \(a_h\), interpolation \eqref{eq:interpol-energy}, and the energy error estimate Theorem~\ref{thm:energy-error}, yielding
\begin{align}
|a_h(e_h,\phi-\pi_h\phi^e)|
&\lesssim
\tn e_h\tn_h \tn \phi^e-\pi_h\phi^e \tn_h
\label{eq:l2-cont-a}
\\
&\lesssim
h \tn e_h\tn_h \|\phi\|_{H^2(\Omega_\delta)}
\label{eq:l2-cont-a-2}
\\
&\lesssim
\Bigl(
h^{p+1} \|u\|_{H^{p+1}(\Omega)}
+
h^{r} \|u^e\|_{H^r(U_\delta(\partial\Omega))}
\\&\qquad\qquad\qquad\quad \nonumber
+
h^{1/2}\delta \|u^e\|_{W^1_\infty(U_\delta(\partial\Omega))}
\Bigr)
\|\phi\|_{H^2(\Omega_\delta)}
\label{eq:l2-cont-a-final}
\end{align}
where $r = \max(p+1,2+\epsilon)$.

For the second term in \eqref{eq:l2-parts}, we get
\begin{align}
a_h(e_h,\pi_h\phi^e)
&=
a_h(u^e,\pi_h\phi^e) - a_h(u_h,\pi_h\phi^e) - s_h(u_h,\pi_h\phi^e) + s_h(u_h,\pi_h\phi^e)
\\&=
a_h(u^e,\pi_h\phi^e) -l_h(\pi_h\phi^e) + s_h(u_h,\pi_h\phi^e)
\\&=
\underbrace{A_h(u^e,\pi_h\phi^e) -l_h(\pi_h\phi^e)}_{\bigstar} + s_h(u_h - u^e,\pi_h\phi^e)
\label{eq:l2-parts-2}
\end{align}
where we for the last term use the weak consistency \eqref{eq:weak-consistency}, interpolation \eqref{eq:interpol-energy}, and the energy error estimate Theorem~\ref{thm:energy-error}, yielding
\begin{align}
|s_h(u_h - u^e,\pi_h\phi^e)|
&\le
\|u_h - u^e\|_{s_h} \|\pi_h\phi^e\|_{s_h}
\\
&\lesssim
\tn u_h - u^e \tn_h  h \|\phi\|_{H^2(\Omega_\delta)}
\\
&\lesssim
\Bigl(
h^{p+1} \|u\|_{H^{p+1}(\Omega)}
+
h^{r} \|u^e\|_{H^r(U_\delta(\partial\Omega))}
\\&\qquad\qquad\qquad\qquad \nonumber
+
h^{1/2}\delta \|u^e\|_{W^1_\infty(U_\delta(\partial\Omega))}
\Bigr)
\|\phi\|_{H^2(\Omega_\delta)}
\end{align}

To estimate the consistency term $\bigstar$ in \eqref{eq:l2-parts-2},
we make use of the same argument as in the proof of the improved $H^1$ estimate, integrating by parts and splitting it into three terms
\begin{align}
\bigstar&=A_h(u^e,\pi_h\phi^e) - l_h(\pi_h\phi^e)  
\\
&=
- \underbrace{(u^e, \Sigma_{n_\delta}(\pi_h\phi^e))_{\partial \Omega_\delta}}_{I} 
- \underbrace{(\Delta u^e + f_\delta,\pi_h\phi^e)_{\Omega_\delta}}_{II}
+ \underbrace{s_h(\pi_h u^e,\pi_h\phi^e)}_{III}
\label{eq:consistency-dual}
\end{align}
The key property we will utilize is that the dual solution $\phi$ is zero on $\partial\Omega_\delta$.

\proofparagraph{Term $I$.}
Using the definition of the Nitsche flux \eqref{eq:flux}, we split
\begin{align}
|I|
&\leq
|(u^e,\nabla_{n_\delta} \pi_h\phi^e)_{\partial \Omega_\delta}|
+
\beta h^{-1}|(u^e,\pi_h\phi^e)_{\partial \Omega_\delta}|
\label{eq:l2-termI-split}
\end{align}
Since \(\phi = 0\) on \(\partial\Omega_\delta\), we have
\begin{equation}
\pi_h\phi^e = \pi_h\phi^e - \phi
\qquad \text{on } \partial\Omega_\delta
\label{eq:l2-dual-zero-boundary}
\end{equation}
Hence, using interpolation, trace inequalities, and \eqref{eq:u-boundary-shift}, we obtain
\begin{align}
h^{-1}|(u^e,\pi_h\phi^e)_{\partial \Omega_\delta}|
&=
h^{-1}|(u^e,\pi_h\phi^e-\phi)_{\partial \Omega_\delta}|
\\
&\lesssim
h^{-1/2}\|u^e\|_{\partial \Omega_\delta}
h^{-1/2}\|\pi_h\phi^e-\phi\|_{\partial \Omega_\delta}
\\
&\lesssim
\delta h^{1/2}
\|u^e\|_{W^1_\infty(U_\delta(\partial\Omega))}
\|\phi\|_{H^2(\Omega_\delta)}
\label{eq:l2-termI-penalty}
\end{align}
For the normal derivative term, adding and subtracting \(\phi\), we get
\begin{align}
|(u^e,\nabla_{n_\delta} \pi_h\phi^e)_{\partial \Omega_\delta}|
&\leq
|(u^e,\nabla_{n_\delta}(\pi_h\phi^e-\phi))_{\partial \Omega_\delta}|
+
|(u^e,\nabla_{n_\delta}\phi)_{\partial \Omega_\delta}|
\label{eq:l2-termI-normal-split}
\end{align}
Using interpolation and trace inequalities for the first term, and a trace estimate for the second term, we obtain
\begin{align}
|(u^e,\nabla_{n_\delta}(\pi_h\phi^e-\phi))_{\partial \Omega_\delta}|
&\lesssim
\|u^e\|_{\partial \Omega_\delta}
\|\nabla_{n_\delta}(\pi_h\phi^e-\phi)\|_{\partial \Omega_\delta}
\\
&\lesssim
\delta h^{1/2}
\|u^e\|_{W^1_\infty(U_\delta(\partial\Omega))}
\|\phi\|_{H^2(\Omega_\delta)}
\label{eq:l2-termI-normal-interp}
\end{align}
and
\begin{align}
|(u^e,\nabla_{n_\delta}\phi)_{\partial \Omega_\delta}|
&\lesssim
\|u^e\|_{\partial \Omega_\delta}
\|\nabla_{n_\delta}\phi\|_{\partial \Omega_\delta}
\\
&\lesssim
\delta
\|u^e\|_{W^1_\infty(U_\delta(\partial\Omega))}
\|\phi\|_{H^2(\Omega_\delta)}
\label{eq:l2-termI-normal-dual}
\end{align}
Combining \eqref{eq:l2-termI-split}, \eqref{eq:l2-termI-penalty}, \eqref{eq:l2-termI-normal-split}, \eqref{eq:l2-termI-normal-interp}, and \eqref{eq:l2-termI-normal-dual}, we conclude that
\begin{equation}
|I|
\lesssim
\delta
\|u^e\|_{W^1_\infty(U_\delta(\partial\Omega))}
\|\phi\|_{H^2(\Omega_\delta)}
\label{eq:l2-termI-final}
\end{equation}

\proofparagraph{Term $II$.}
Since \(-\Delta u = f\) in \(\Omega\), the residual is supported in the domain mismatch region, and therefore
\begin{equation}
|II|
\leq
\|\Delta u^e + f_\delta\|_{\Omega_\delta\setminus\Omega}
\|\pi_h\phi^e\|_{\Omega_\delta\setminus\Omega}
\label{eq:l2-termII-start}
\end{equation}
We next estimate the dual factor. Using the domain mismatch estimate Lemma~\ref{lem:strip} together with \eqref{eq:l2-dual-zero-boundary}, interpolation, and elliptic regularity, we obtain
\begin{align}
\|\pi_h\phi^e\|_{\Omega_\delta\setminus\Omega}
&\lesssim
\delta^{1/2}\|\pi_h\phi^e\|_{\partial\Omega_\delta}
+
\delta\|\nabla \pi_h\phi^e\|_{\Omega_\delta}
\\
&=
\delta^{1/2}\|\pi_h\phi^e-\phi\|_{\partial\Omega_\delta}
+
\delta\|\nabla \pi_h\phi^e\|_{\Omega_\delta}
\\
&\lesssim
(\delta^{1/2}h^{3/2}+\delta)
\|\phi\|_{H^2(\Omega_\delta)}
\label{eq:l2-termII-dual-strip}
\end{align}
Using the residual estimate from the proof of the standard energy estimate with \(s=\max(p-1,\epsilon)\), we get
\begin{align}
\|\Delta u^e + f_\delta\|_{\Omega_\delta\setminus\Omega}
&\lesssim
\delta^s \|\Delta u^e + f_\delta\|_{H^s(U_\delta(\partial\Omega))}
\lesssim
\delta^s \|u^e\|_{H^{s+2}(U_\delta(\partial\Omega))}
\label{eq:l2-termII-residual}
\end{align}
Combining \eqref{eq:l2-termII-start}, \eqref{eq:l2-termII-dual-strip}, and \eqref{eq:l2-termII-residual}, and using \(\delta\lesssim h\), we obtain
\begin{align}
|II|
&\lesssim
(\delta^{s+1/2}h^{3/2}+\delta^{s+1})
\|u^e\|_{H^{s+2}(U_\delta(\partial\Omega))}
\|\phi\|_{H^2(\Omega_\delta)}
\\
&\lesssim
(h^{s+2}+\delta^{s+1})
\|u^e\|_{H^{s+2}(U_\delta(\partial\Omega))}
\|\phi\|_{H^2(\Omega_\delta)}
\label{eq:l2-termII-final}
\end{align}

\proofparagraph{Term $III$.}
Using Cauchy--Schwarz, the weak consistency estimate \eqref{eq:weak-consistency} for the primal and dual variables, and elliptic regularity, we get
\begin{align}
|III|
&=
|s_h(\pi_h u^e,\pi_h\phi^e)|
\\
&\leq
\|\pi_h u^e\|_{s_h}
\|\pi_h\phi^e\|_{s_h}
\\
&\lesssim
h^{p+1} \|u\|_{H^{p+1}(\Omega)}
\|\phi\|_{H^2(\Omega_\delta)}
\label{eq:l2-termIII-final}
\end{align}

Collecting the estimates \eqref{eq:l2-termI-final}, \eqref{eq:l2-termII-final}, and \eqref{eq:l2-termIII-final} in \eqref{eq:consistency-dual}, we obtain
\begin{equation}
|\bigstar|
\lesssim
\Bigl(
h^{p+1}
\|u\|_{H^{p+1}(\Omega)}
+
(h^r+\delta^{r-1})
\|u^e\|_{H^{r}(U_\delta(\partial\Omega))}
+
\delta
\|u^e\|_{W^1_\infty(U_\delta(\partial\Omega))}
\Bigr)
\|\phi\|_{H^2(\Omega_\delta)}
\label{eq:l2-consistency-final}
\end{equation}
where \(r = s+2 = \max(p+1,2+\epsilon)\).
Substituting this bound and the estimate of the stabilization term in \eqref{eq:l2-parts-2}, and then using \eqref{eq:l2-cont-a-final} and the dual regularity estimate \eqref{eq:l2-dual-regularity}, yields
\begin{align}
\label{eq:l2-before-division}
\|e_h\|_{\Omega_\delta}^2
&\lesssim
\Bigl(
h^{p+1} \|u\|_{H^{p+1}(\Omega)}
+
(h^{r}+\delta^{r-1}) \|u^e\|_{H^r(U_\delta(\partial\Omega))}
\\&\qquad\qquad\qquad\qquad \nonumber
+
(\delta + h^{1/2}\delta) \|u^e\|_{W^1_\infty(U_\delta(\partial\Omega))}
\Bigr)
\|e_h\|_{\Omega_\delta}
\end{align}
Dividing by \(\|e_h\|_{\Omega_\delta}\) and using \(h^{1/2}\delta \leq \delta\) completes the proof of \eqref{eq:l2-error}.
\end{proof}

\section{Numerical Examples}
\label{sec:numer-examples}

We present numerical examples that illustrate the error estimates derived in the previous sections. The examples demonstrate the convergence behavior for different polynomial orders and error norms, as well as the effect of boundary perturbation and normal approximation errors.

\paragraph{Summary of Geometric Scalings.}
The theoretical results show how the geometric approximation errors enter the
different error estimates. Consequently, they provide sufficient scalings of the
geometry errors to guarantee optimal order convergence in the corresponding
norms.
\begin{itemize}
\item For the energy norm, the boundary location error $\delta$ enters with an
\(h^{-1/2}\)-amplification. The estimate guarantees optimal order
convergence of order \(h^p\) if
\(\delta \lesssim h^{p+1/2}\), and there is no explicit contribution from the normal
approximation error \(\delta_n\).

\item For the \(H^1\)-seminorm, the boundary location and normal approximation
errors enter additively. The estimate guarantees optimal order
convergence of order \(h^p\) if
\(\delta \lesssim h^p\) and \(\delta_n \lesssim h^p\).

\item For the \(L^2\)-norm, the geometry contribution $\delta$ appears as a separate
additive term. The estimate guarantees optimal order convergence of
order \(h^{p+1}\) if
\(\delta \lesssim h^{p+1}\), and there is no explicit contribution from the normal
approximation error \(\delta_n\).
\end{itemize}

\paragraph{Implementation and Parameter Choices.}
The method is implemented in MATLAB in two space dimensions and the linear system of equations is solved using a direct solver (MATLAB's \verb+\+ operator), and we refer to \cite{MR3682761} for details on implementation.

In all experiments we use tensor product Lagrange basis functions of degree $p=1,2,3$ on a uniform background grid of mesh size $h$. The geometry of the perturbed domain $\Omega_\delta$ is described as a high resolution polygon.

For the penalty parameter $\beta$ we choose a fixed value $\beta=25p^2$, and the ghost penalty stabilization parameter as $\gamma_j = \frac{0.01}{((j-1)!)^2 j}$.

\begin{rem}[Beyond the Smooth Setting]
The analysis assumes that both the exact and approximate boundaries are smooth. In the numerical experiments, however, we also consider examples with nonsmooth or only piecewise smooth geometries, such as the square domain in the $\delta$-scaling study and the polygonal boundary arising from the piecewise linear level-set approximation. These experiments are included to illustrate the practical behavior of the method beyond the setting fully covered by the theory.
\end{rem}

\subsection{$\delta$-Scaling Study}

We consider the unit square domain $\Omega = (0,1)^2$ with constant loading function $f = 1$ and homogeneous Dirichlet boundary conditions. One representation of the solution to this problem is the series expansion
\begin{align} \label{eq:series-expansion}
u(x,y) &= u_p(x,y) - \sum_{\substack{n=1 \\ n \,\text{odd}}}^{\infty}
\frac{2}{\pi^3 n^3}
\big( S_n(y)\sin(n\pi x) + S_n(x)\sin(n\pi y)\big)
\end{align}
where $u_p(x,y)$, a particular solution, respectively $S_n(t)$ are given by
\begin{align}
	u_p(x,y) = \frac{x(1-x) + y(1-y)}{4}
	\,,\quad
	S_n(t) &= \frac{\sinh(n\pi(1-t)) + \sinh(n\pi t)}{\sinh(n\pi)}
\end{align}
Since the series \eqref{eq:series-expansion} converges rapidly, we truncate the series after 50 terms (for $p=1,2$) and 100 terms (for $p=3$) and obtain a sufficiently precise approximation of the solution.

Let $(r,\theta)$ denote the polar coordinates of $x \in \mathbb{R}^2$ with respect to the center $x_0=(0.45,0.35)$. 
We define the boundary perturbation $\widetilde{q}_\delta:\partial\Omega \to \partial\Omega_\delta$ by the radial map
\begin{align}
\widetilde{q}_\delta(x) = x + \delta \cos(5\theta)\,\hat{r}
\end{align}
where $\hat{r} = (x - x_0)/|x - x_0|$ is the radial unit vector. 
In this case, the normal approximation error satisfies $\delta_n \sim \delta$.
The perturbed domain $\Omega_\delta$, together with the mesh and numerical solutions for two mesh sizes $h_0$ and $h_1=h_0/2$, is shown in Figure~\ref{fig:delta_study_mesh_solutions}.

In Figure~\ref{fig:delta_study}, we study the convergence in the energy norm, the $H^1$-seminorm, and the $L^2$-norm for polynomial orders $p=1,2,3$ and different scalings $\delta = h^{\alpha}$. 
The results clearly demonstrate that optimal order convergence is obtained when the scalings predicted by the analysis are satisfied, namely $\delta \sim h^{p+1/2}$ for the energy norm, $\delta \sim h^{p}$ for the $H^1$-seminorm, and $\delta \sim h^{p+1}$ for the $L^2$-norm.

\begin{figure}
\centering
\begin{subfigure}[b]{0.32\textwidth}
\centering
\includegraphics[width=\textwidth]{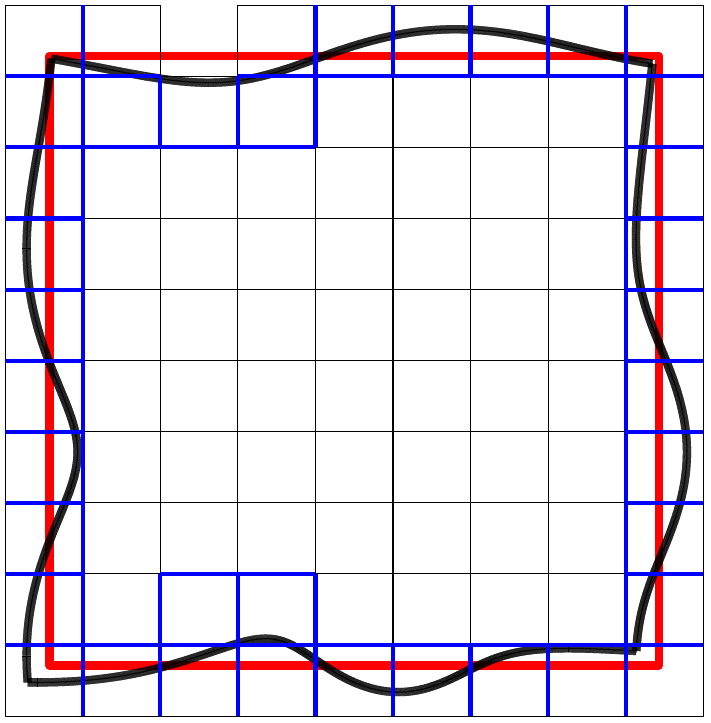}
\caption{$h_0$ mesh}
\label{fig:delta_study_h0_mesh}
\end{subfigure}
\hfill
\begin{subfigure}[b]{0.32\textwidth}
\centering
\includegraphics[width=\textwidth,trim={262 88 233 70},clip]{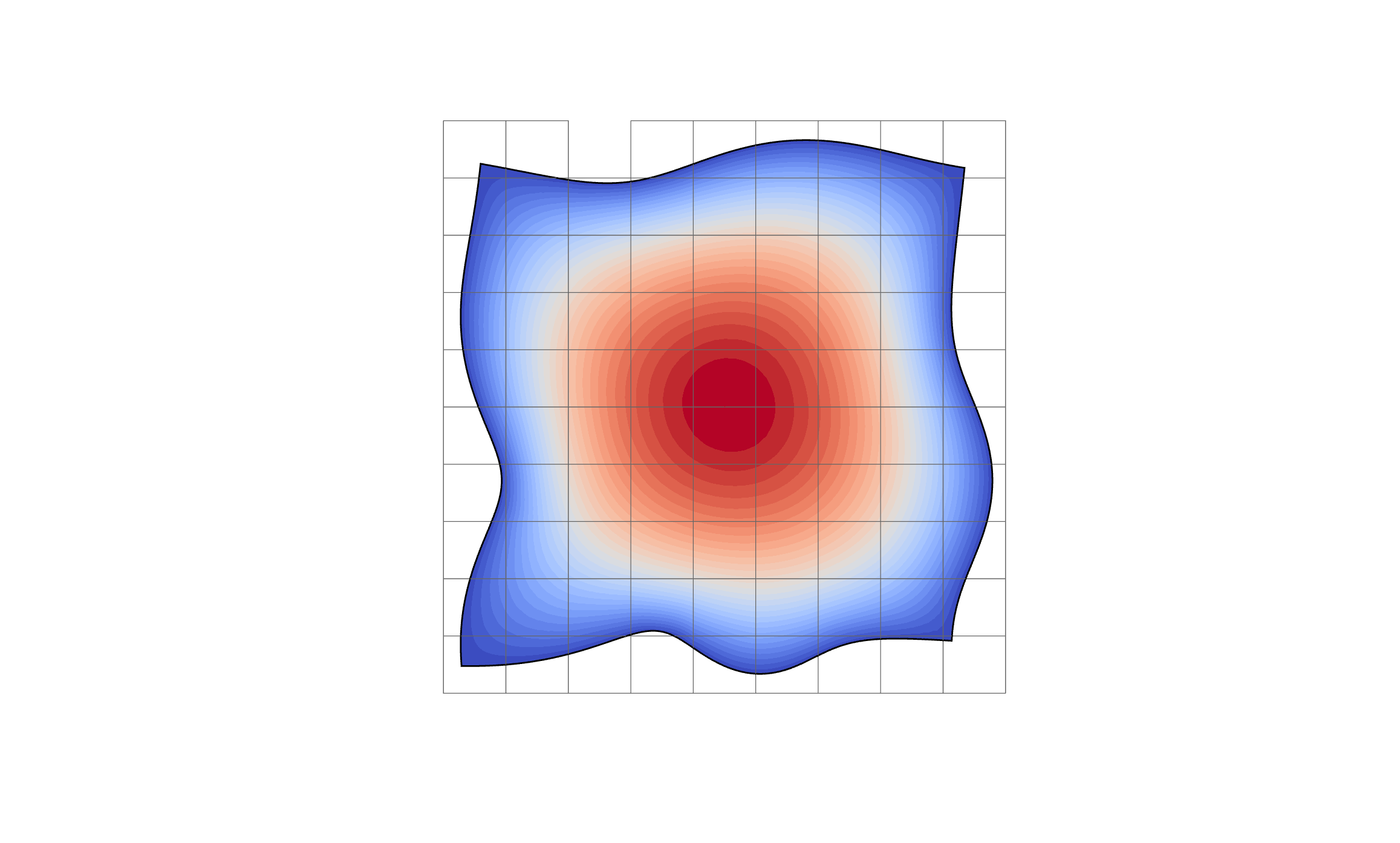}
\caption{$h_0$ solution}
\label{fig:delta_study_h0_solution}
\end{subfigure}
\hfill
\begin{subfigure}[b]{0.32\textwidth}
\centering
\includegraphics[width=\textwidth,trim={262 88 233 70},clip]{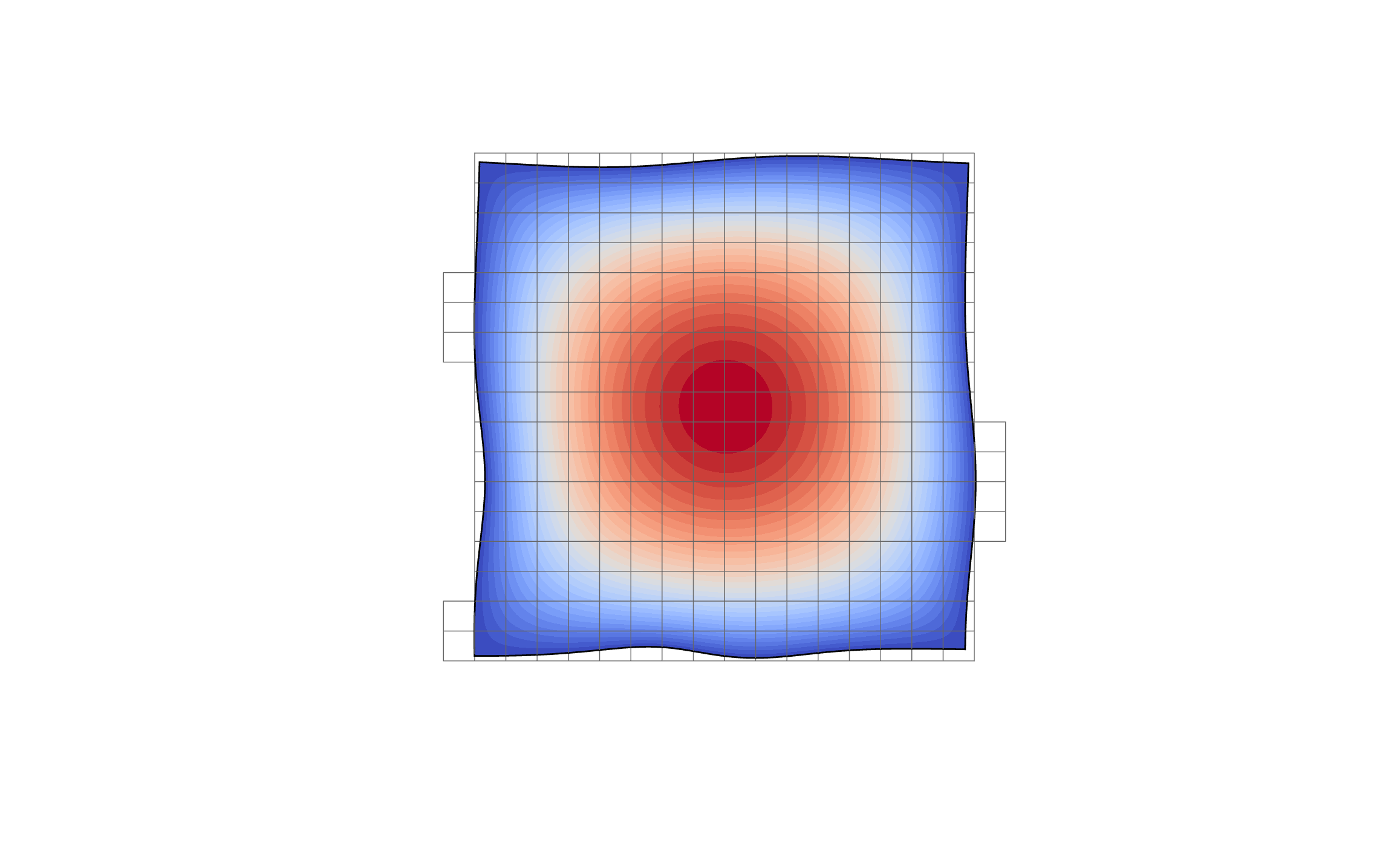}
\caption{$h_1$ solution, $\delta=h^{p+1/2}$}
\label{fig:delta_study_h1_solution}
\end{subfigure}

\caption{\emph{$\delta$-study mesh and solutions.} Illustration of the mesh and numerical solution for two mesh sizes, $h_0$ and $h_1=h_0/2$, when the boundary perturbation is chosen as $\delta = h^{p+1/2}$. The scaling of $\delta$ with $h$ results in different perturbed domains on each mesh.}
\label{fig:delta_study_mesh_solutions}
\end{figure}

\begin{figure}
\centering
\begin{subfigure}[b]{0.32\textwidth}
\centering
\includegraphics[width=\textwidth]{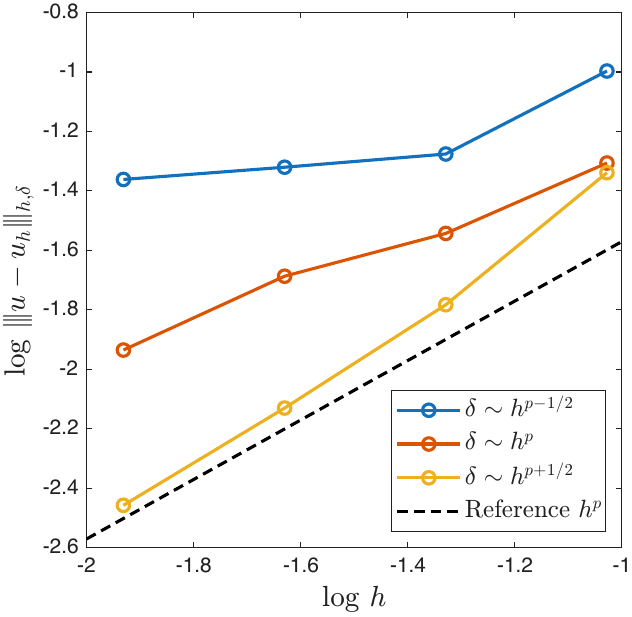}
\caption{$p=1$, Energy norm}
\label{fig:delta_p1_E}
\end{subfigure}
\hfill
\begin{subfigure}[b]{0.32\textwidth}
\centering
\includegraphics[width=\textwidth]{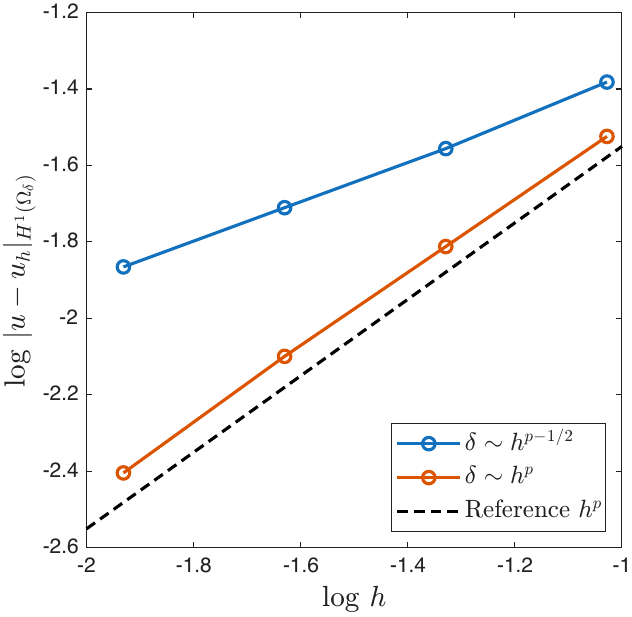}
\caption{$p=1$, $H^1$-seminorm}
\label{fig:delta_p1_H1}
\end{subfigure}
\hfill
\begin{subfigure}[b]{0.32\textwidth}
\centering
\includegraphics[width=\textwidth]{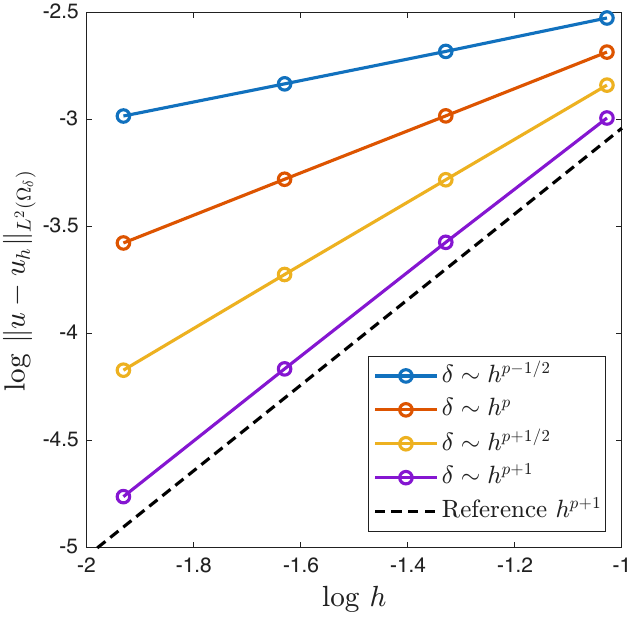}
\caption{$p=1$, $L^2$-norm}
\label{fig:delta_p1_L2}
\end{subfigure}

\vspace{0.5cm}

\begin{subfigure}[b]{0.32\textwidth}
\centering
\includegraphics[width=\textwidth]{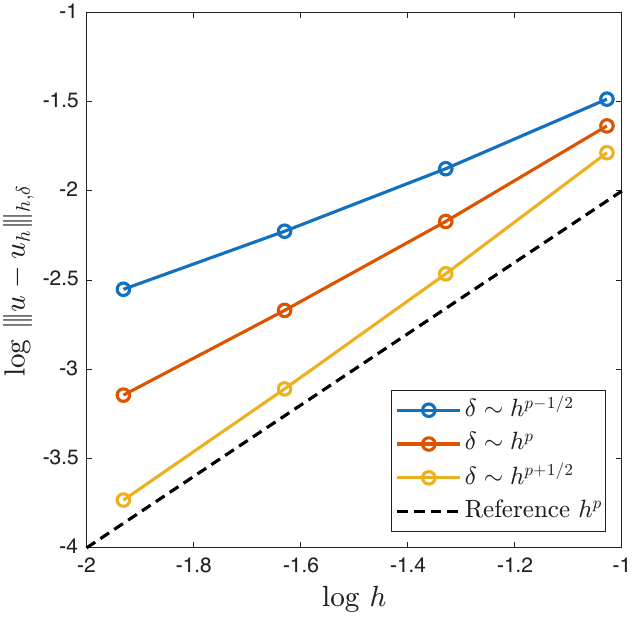}
\caption{$p=2$, Energy norm}
\label{fig:delta_p2_E}
\end{subfigure}
\hfill
\begin{subfigure}[b]{0.32\textwidth}
\centering
\includegraphics[width=\textwidth]{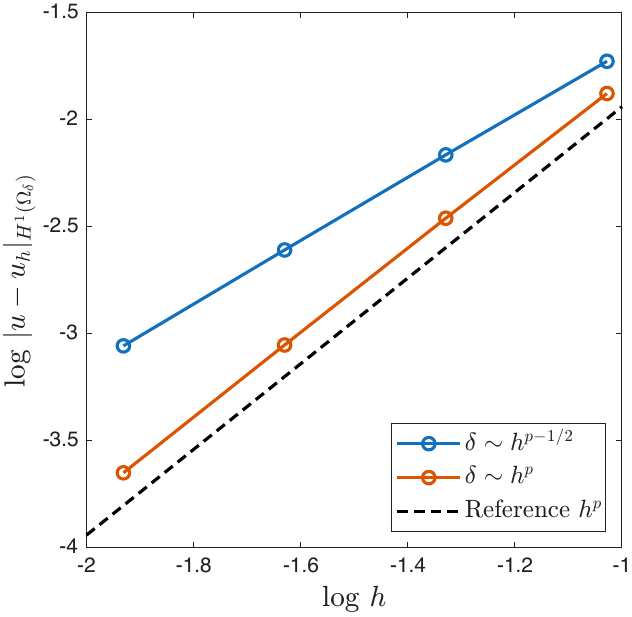}
\caption{$p=2$, $H^1$-seminorm}
\label{fig:delta_p2_H1}
\end{subfigure}
\hfill
\begin{subfigure}[b]{0.32\textwidth}
\centering
\includegraphics[width=\textwidth]{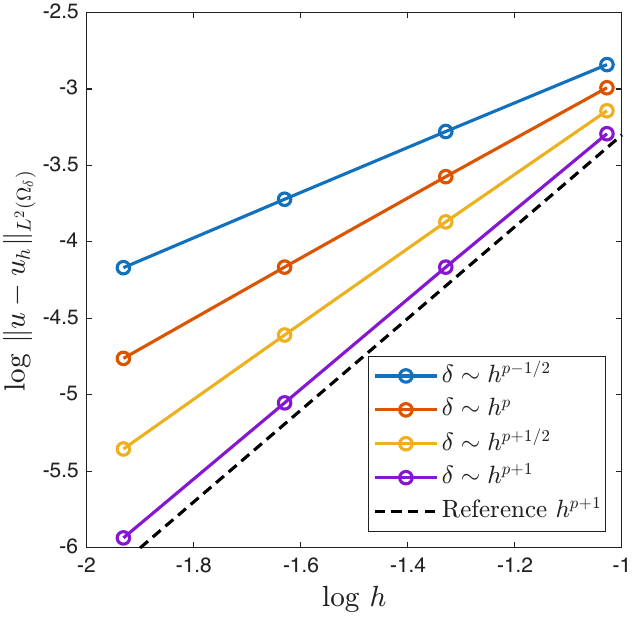}
\caption{$p=2$, $L^2$-norm}
\label{fig:delta_p2_L2}
\end{subfigure}

\vspace{0.5cm}

\begin{subfigure}[b]{0.32\textwidth}
\centering
\includegraphics[width=\textwidth]{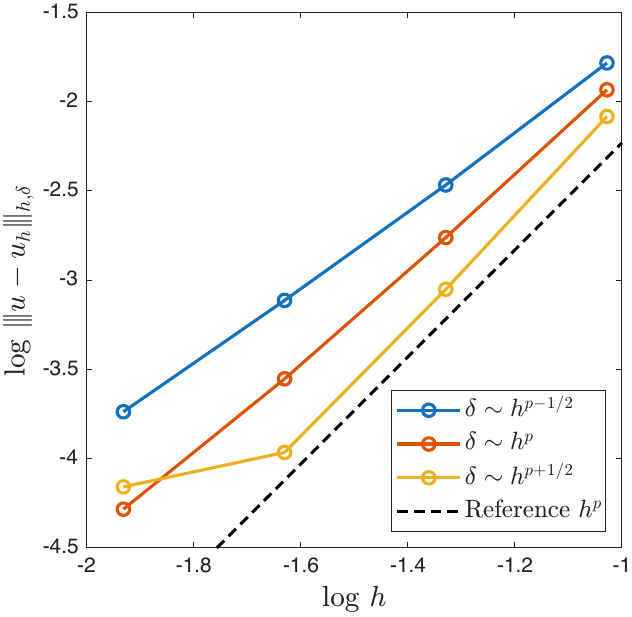}
\caption{$p=3$, Energy norm}
\label{fig:delta_p3_E}
\end{subfigure}
\hfill
\begin{subfigure}[b]{0.32\textwidth}
\centering
\includegraphics[width=\textwidth]{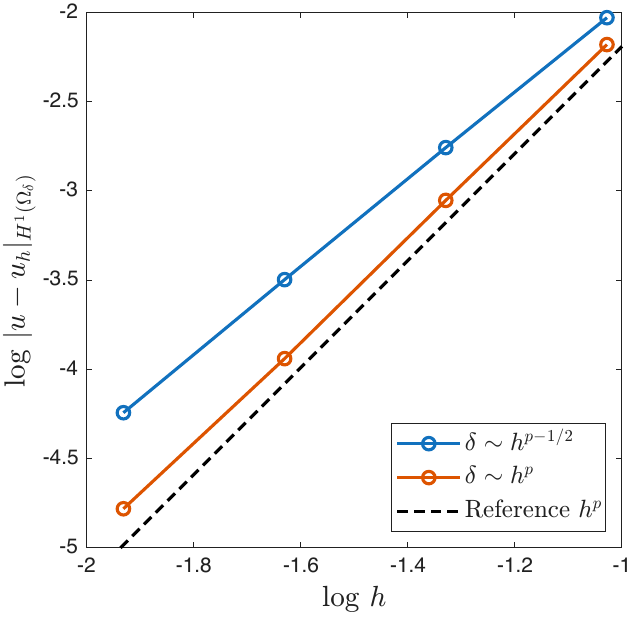}
\caption{$p=3$, $H^1$-seminorm}
\label{fig:delta_p3_H1}
\end{subfigure}
\hfill
\begin{subfigure}[b]{0.32\textwidth}
\centering
\includegraphics[width=\textwidth]{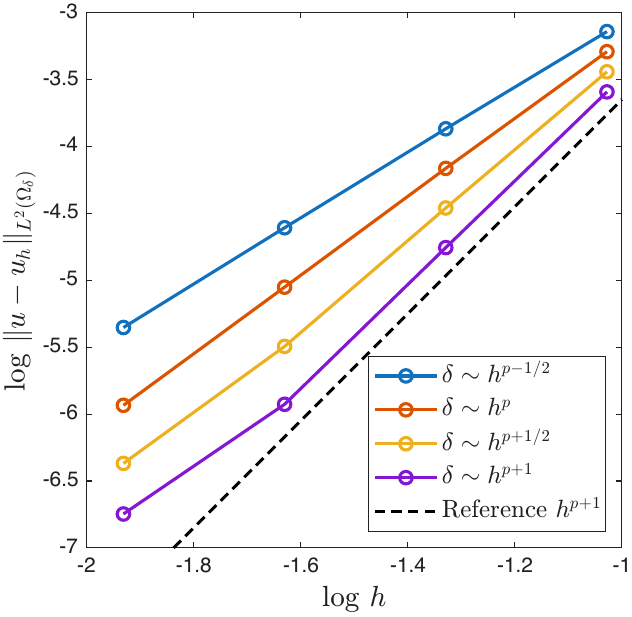}
\caption{$p=3$, $L^2$-norm}
\label{fig:delta_p3_L2}
\end{subfigure}

\caption{\emph{$\delta$-scaling.} Convergence in the energy norm, $H^1$-seminorm, and $L^2$-norm for polynomial orders $p=1,2,3$ and different scalings $\delta = h^\alpha$. The results illustrate how the choice of $\alpha$ affects the convergence rate, and confirm that optimal order convergence is obtained when $\delta$ scales according to the theoretical predictions for each norm. The slight loss of convergence rate observed for $p=3$ in the energy and $L^2$ norms is attributed to numerical inaccuracies in the evaluation of the series expansion used to compute the reference solution.}
\label{fig:delta_study}
\end{figure}

\subsection{Normal Approximation Study}

We consider the unit circle domain $\Omega$ and the solution $u$ to
\begin{align}
	-\Delta u = 1 \quad \text{in $\Omega$}, \qquad u = 0 \quad \text{on $\partial\Omega$}
\end{align}
Since this problem has the analytical solution $u = \frac{1}{4}(1-x^2-y^2)$, which is a second order polynomial, the interpolation error vanishes for $p\geq 2$. We use $p=2$ finite elements on the perturbed boundary problem and hence in this example effectively isolate the geometric error.

We define the boundary perturbation $\widetilde{q}_\delta:\partial\Omega \to \partial\Omega_\delta$ by the radial map
\begin{align}
\widetilde{q}_\delta(x) = x + \delta \cos\!\left(\operatorname{round}\!\bigl( 5(h/h_0)^{-\alpha_n} \bigr)\theta\right) n_{\partial\Omega}
\label{eq:normal-perturbation}
\end{align}
where $\theta$ is the polar angle of $x \in \partial\Omega$ with respect to the origin. 
In this case, the normal approximation error scales as $\delta_n \sim h^{-\alpha_n}\delta$.
The perturbed domain $\Omega_\delta$, together with the mesh and numerical solutions for two mesh sizes $h_0$ and $h_1 = h_0/2$, and $\alpha_n=0,1$, is shown in Figure~\ref{fig:normal_study_mesh_solutions}.

In Figure~\ref{fig:normal_study}, we study the convergence in the energy norm, the $H^1$-seminorm, and the $L^2$-norm. The parameter $\delta$ is chosen according to the scalings implied by the estimatesfor optimal convergence in each respective norm, while varying $\alpha_n$ in \eqref{eq:normal-perturbation} to vary the normal approximation error. In agreement with the analysis, we observe that the energy norm and the $L^2$-norm converge optimally for all choices of $\alpha_n$, while the $H^1$-seminorm converges suboptimally for all tested $\alpha_n>0$. We also observe that the $L^2$-errors decrease slightly with increasing $\alpha_n$, although the convergence rates remain optimal. This may indicate a mild stabilizing effect of the increased oscillation frequency, but a precise explanation is beyond the scope of the present analysis.

\begin{figure}
\centering
\begin{subfigure}[b]{0.32\textwidth}
\centering
\includegraphics[width=\textwidth]{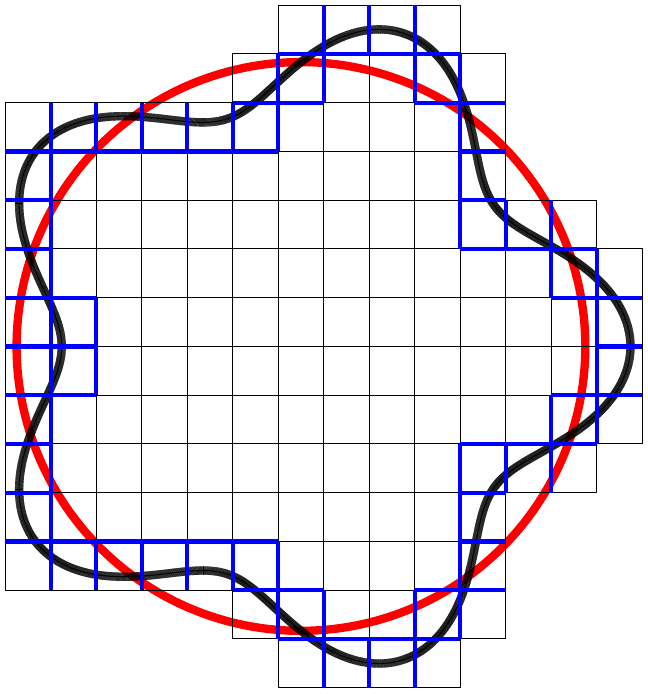}
\caption{$h_0$ mesh}
\label{fig:normal_study_h0_mesh}
\end{subfigure}
\quad
\begin{subfigure}[b]{0.32\textwidth}
\centering
\includegraphics[width=\textwidth,trim={288 95 237 77},clip]{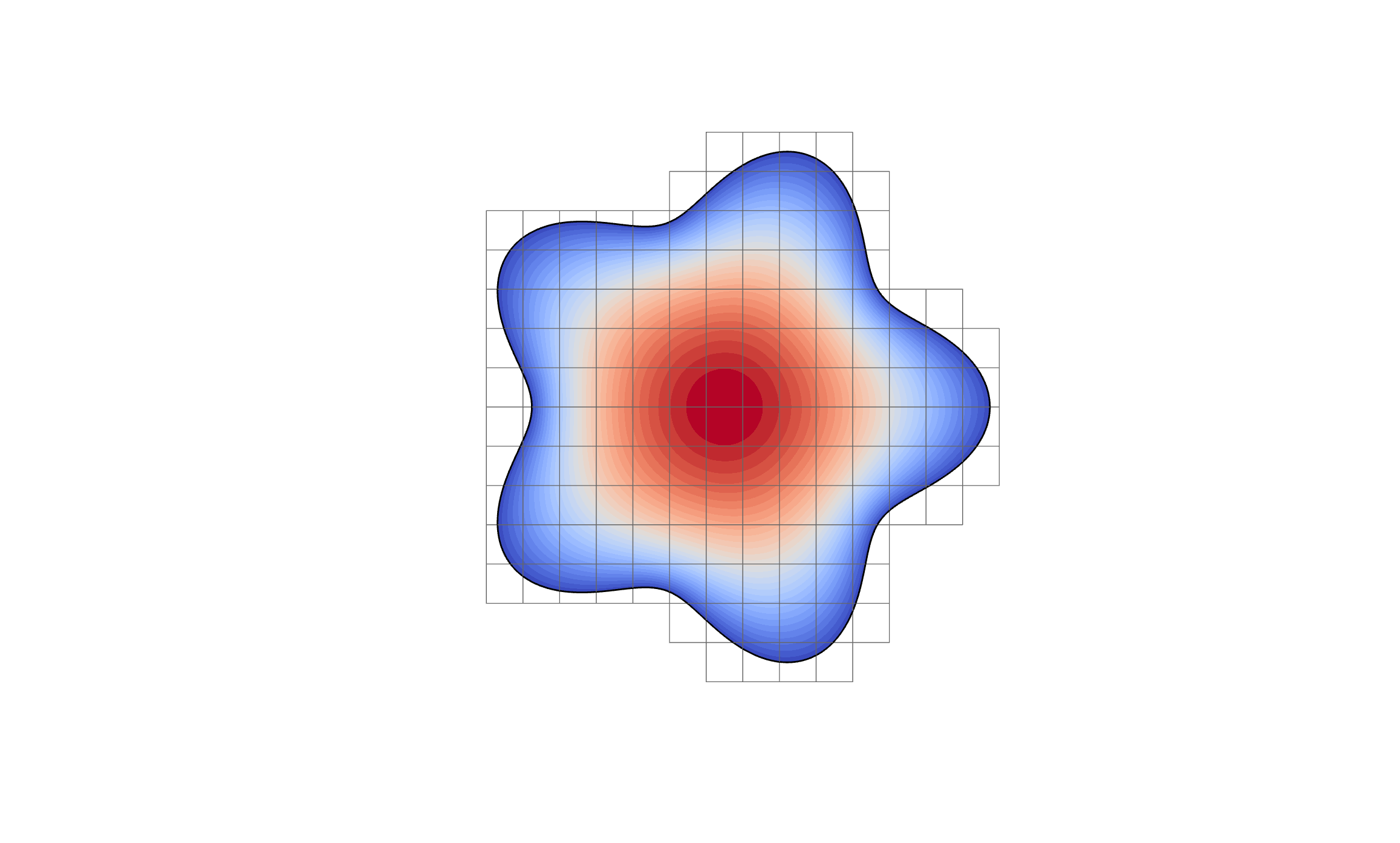}
\caption{$h_0$ solution}
\label{fig:normal_study_h0_solution}
\end{subfigure}

\vspace{0.5cm}

\begin{subfigure}[b]{0.32\textwidth}
\centering
\includegraphics[width=\textwidth,trim={288 95 237 77},clip]{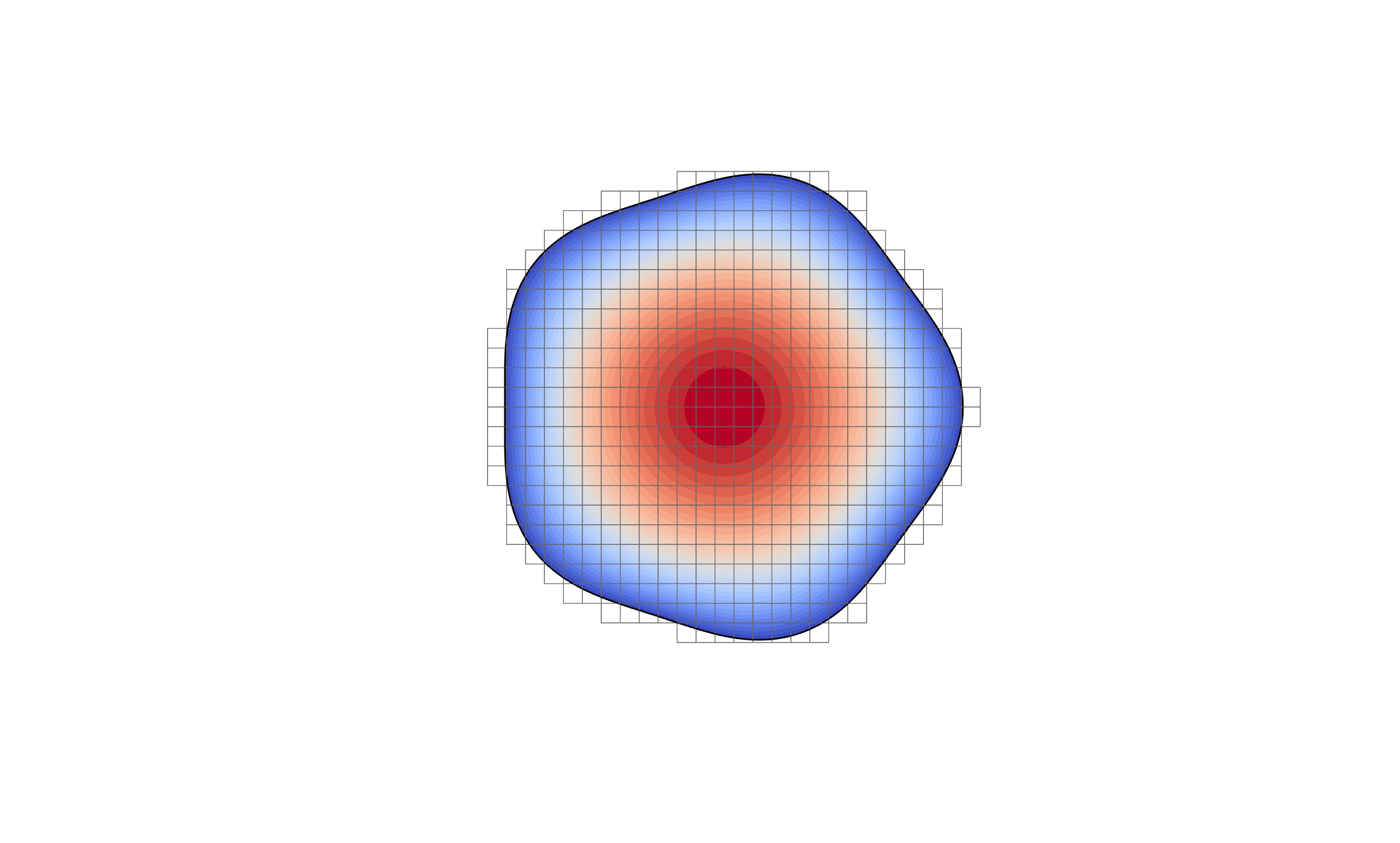}
\caption{$h_1$ solution, $\alpha_n=0$}
\label{fig:normal_study_h1_alpha0}
\end{subfigure}
\quad
\begin{subfigure}[b]{0.32\textwidth}
\centering
\includegraphics[width=\textwidth,trim={288 95 237 77},clip]{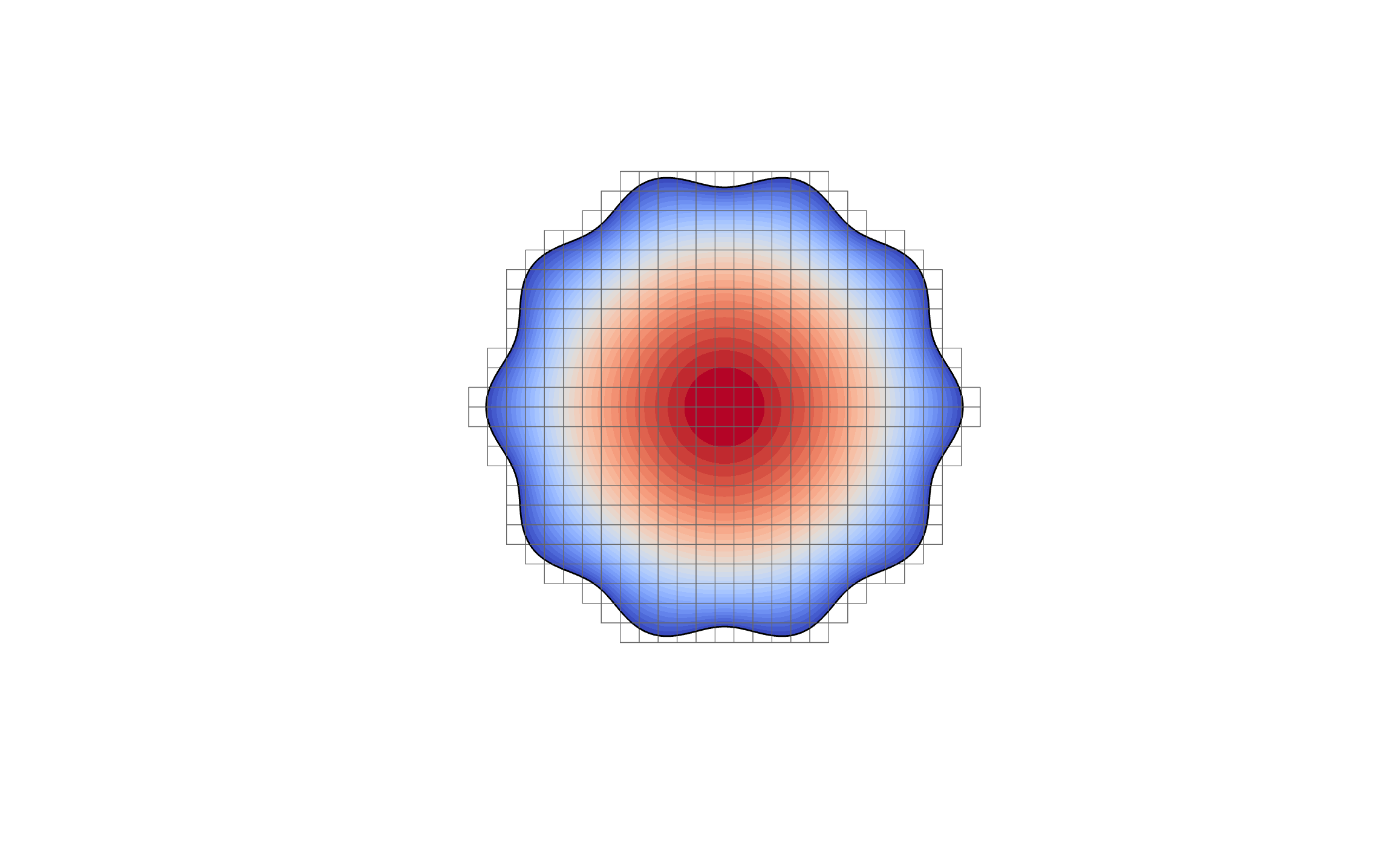}
\caption{$h_1$ solution, $\alpha_n=1$}
\label{fig:normal_study_h1_alpha1}
\end{subfigure}

\caption{\emph{Normal study mesh and solutions.} Illustration of the mesh and numerical solution for two mesh sizes, $h_0$ and $h_1=h_0/2$, for different values of the parameter $\alpha_n$ in \eqref{eq:normal-perturbation}. The boundary perturbation is scaled as $\delta = h^p$, while $\alpha_n$ controls the oscillation frequency and thereby the accuracy of the normal approximation, yielding $\delta_n \sim h^{-\alpha_n}\delta$.}
\label{fig:normal_study_mesh_solutions}
\end{figure}

\begin{figure}
\centering
\begin{subfigure}[b]{0.32\textwidth}
\centering
\includegraphics[width=\textwidth]{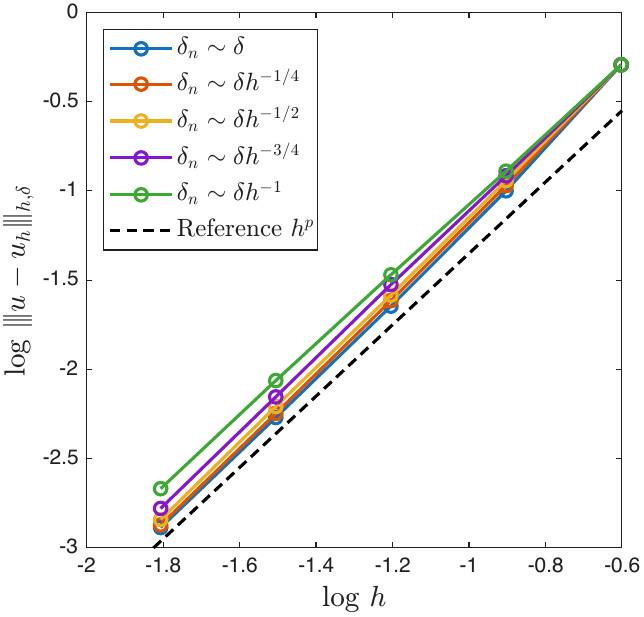}
\caption{Energy norm, $\delta=h^{p+1/2}$}
\label{fig:normal_p2_E}
\end{subfigure}
\hfill
\begin{subfigure}[b]{0.32\textwidth}
\centering
\includegraphics[width=\textwidth]{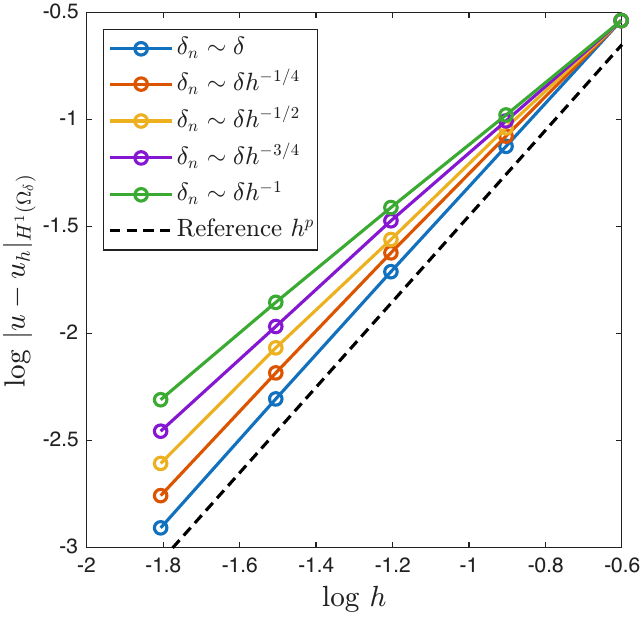}
\caption{$H^1$-seminorm, $\delta=h^{p}$}
\label{fig:normal_p2_H1}
\end{subfigure}
\hfill
\begin{subfigure}[b]{0.32\textwidth}
\centering
\includegraphics[width=\textwidth]{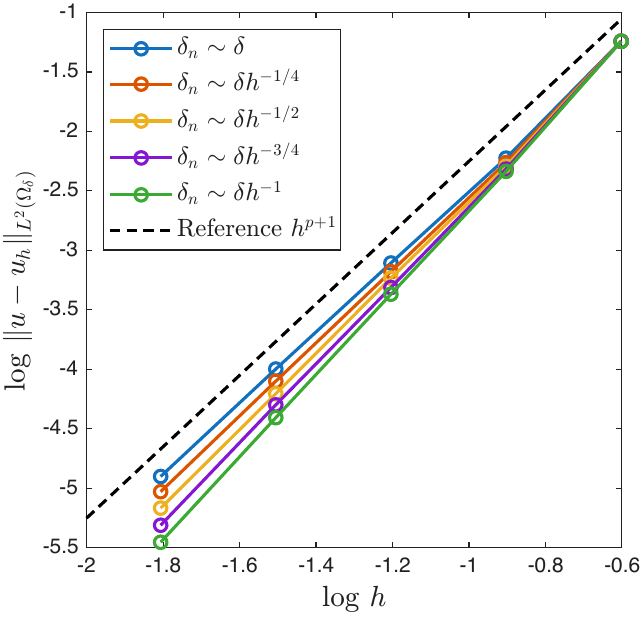}
\caption{$L^2$-norm, $\delta=h^{p+1}$}
\label{fig:normal_p2_L2}
\end{subfigure}

\caption{\emph{Normal study.} Convergence in the energy norm, $H^1$-seminorm, and $L^2$-norm for $p=2$ and varying values of $\alpha_n$. The energy norm and $L^2$-norm retain optimal convergence for all $\alpha_n$, while the $H^1$-seminorm deteriorates for $\alpha_n > 0$, in agreement with the analysis. The $L^2$-errors decrease slightly with increasing $\alpha_n$, which may suggest a mild stabilizing effect of the higher oscillation frequency.}
\label{fig:normal_study}
\end{figure}

\subsection{Level-Set Study}

A common geometric representation in unfitted finite element methods is to use a discrete level-set function $\phi_h: \IR^d \to \IR$ to approximate the boundary $\partial\Omega$ as the zero level set of $\phi_h$, i.e. $\partial\Omega_\delta = \{x \in \IR^d : \phi_h(x) = 0\}$. If $\phi_h$ is piecewise linear, the boundary is described by a polygon, and the geometric errors satisfy $\delta \sim h^2$ and $\delta_n \sim h$, assuming that $\partial\Omega$ is sufficiently smooth.
In Figure~\ref{fig:ls_study_mesh_solutions}, we show the mesh and numerical solutions for two mesh sizes when a piecewise linear level-set function is used to describe the domain.

In Figure~\ref{fig:ls_study}, we study convergence in the energy norm, the $H^1$-seminorm, and the $L^2$-norm using a piecewise linear level-set function. For $p=1$, we observe optimal order convergence in all norms. For $p=2$, the observed convergence rates are consistent with the geometry-induced limitations predicted by the analysis. In particular, the energy norm and the $L^2$-norm exhibit the expected geometry-limited convergence rates, while the $H^1$-seminorm converges suboptimally due to the normal approximation error. However, the observed $H^1$ convergence rate is approximately $h^{1/2}$ higher than predicted by the analysis. 

In light of the normal approximation study, which indicates that the $\delta_n$-term in the $H^1$-estimate is sharp, we attribute this improved behavior to the fact that the worst-case scaling of the normal error is not fully realized in this level-set approximation.

\begin{figure}
\centering
\begin{subfigure}[b]{0.32\textwidth}
\centering
\includegraphics[width=\textwidth]{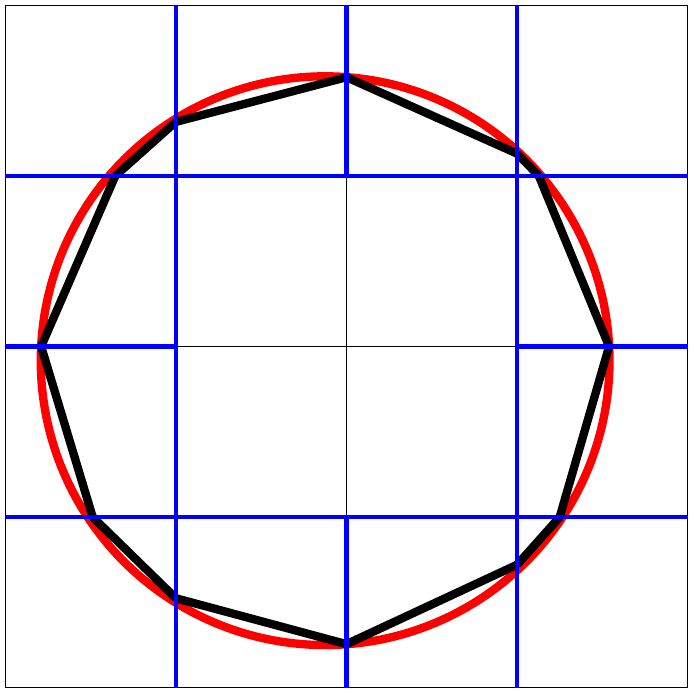}
\caption{$h_0$ mesh}
\label{fig:ls_study_h0_mesh}
\end{subfigure}
\hfill
\begin{subfigure}[b]{0.32\textwidth}
\centering
\includegraphics[width=\textwidth,trim={276 101 227 70},clip]{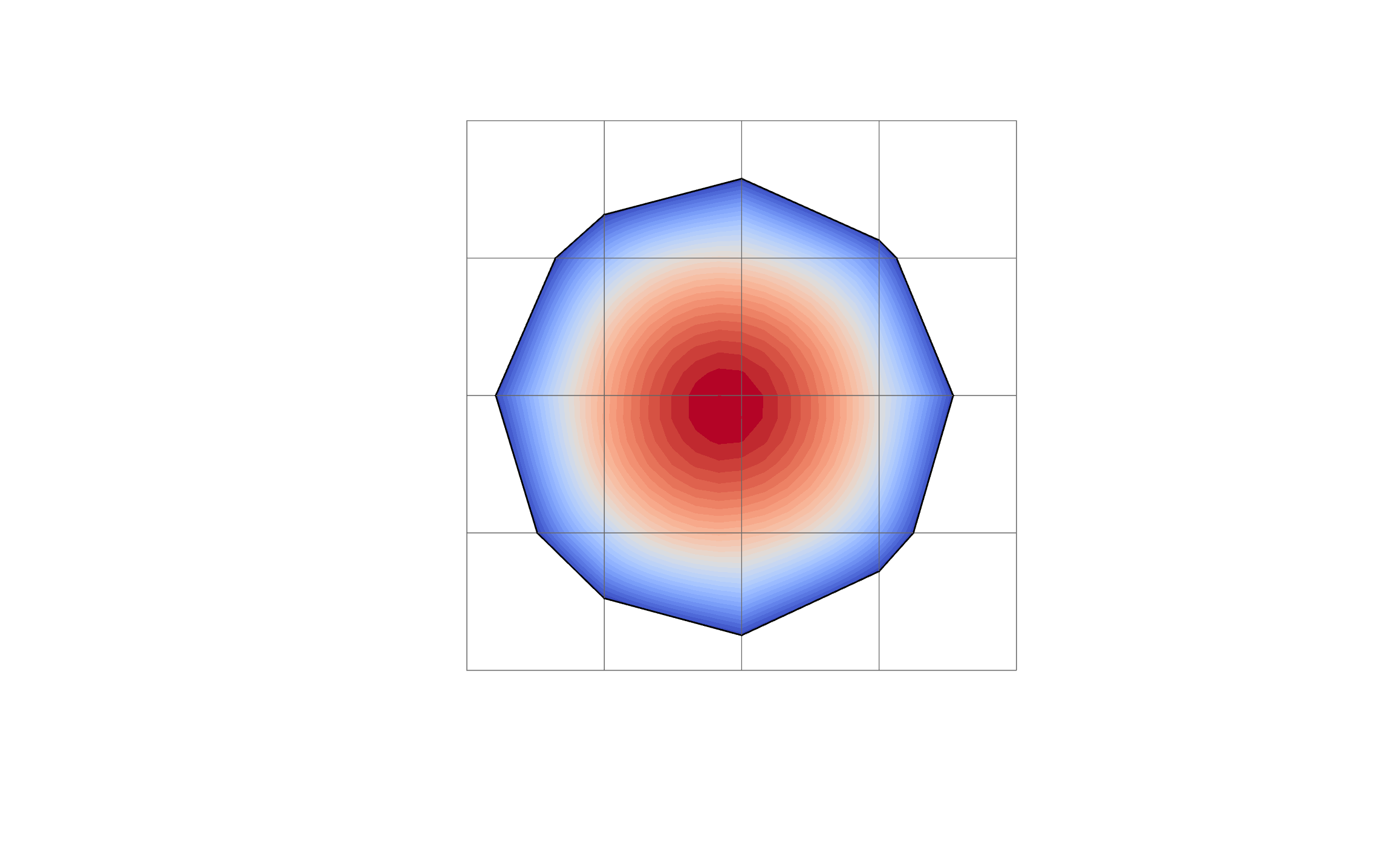}
\caption{$h_0$ solution}
\label{fig:ls_study_h0_solution}
\end{subfigure}
\hfill
\begin{subfigure}[b]{0.32\textwidth}
\centering
\includegraphics[width=\textwidth,trim={276 101 227 70},clip]{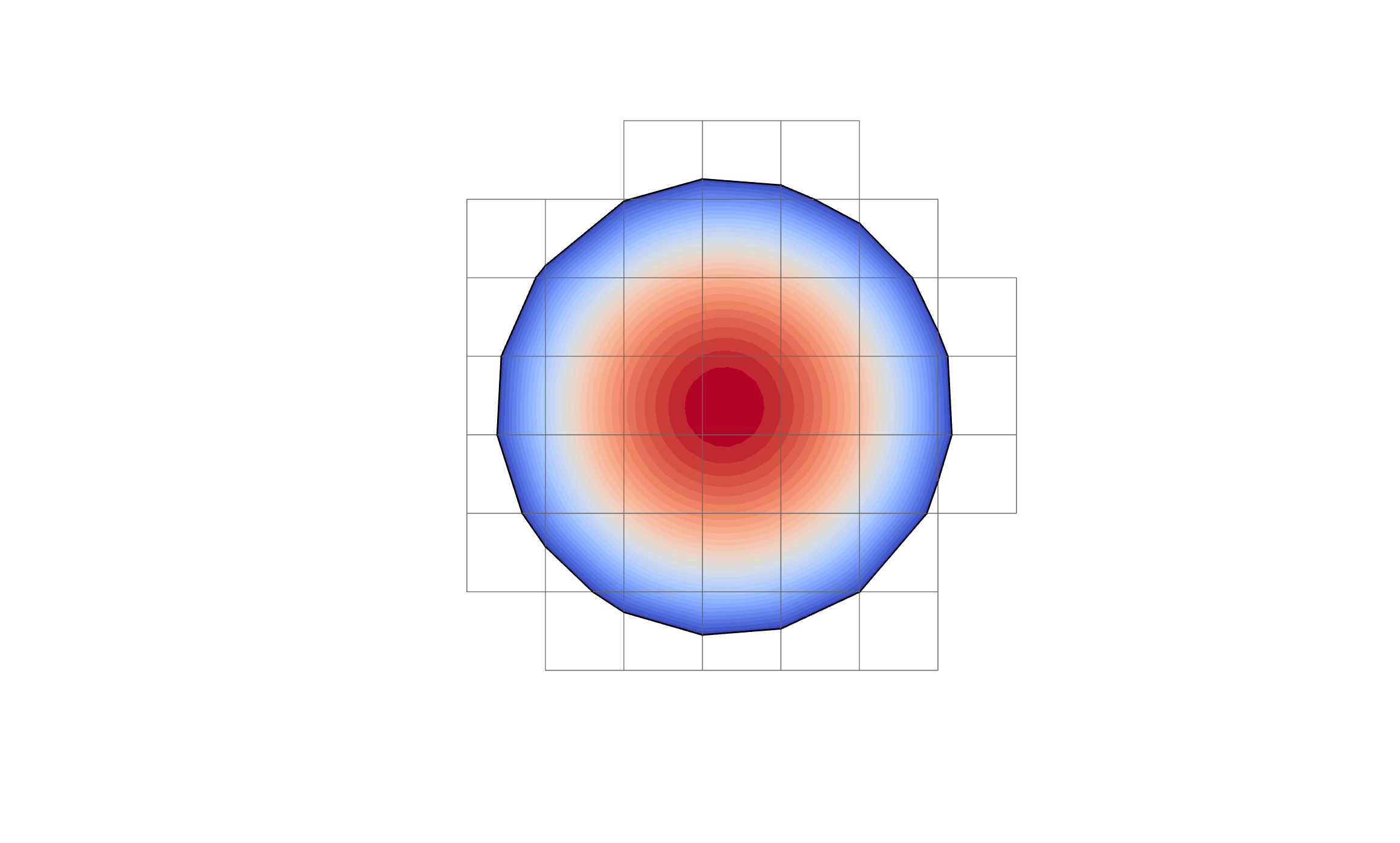}
\caption{$h_1$ solution}
\label{fig:ls_study_h1_solution}
\end{subfigure}

\caption{\emph{Level-set mesh and solutions.} Illustration of the mesh and numerical solution for two mesh sizes, $h_0$ and $h_1=h_0/2$, when a piecewise linear level-set function is used to describe the domain. The resulting polygonal boundary provides a geometric approximation with $\delta \sim h^2$ and $\delta_n \sim h$.}
\label{fig:ls_study_mesh_solutions}
\end{figure}

\begin{figure}
\centering
\begin{subfigure}[b]{0.32\textwidth}
\centering
\includegraphics[width=\textwidth]{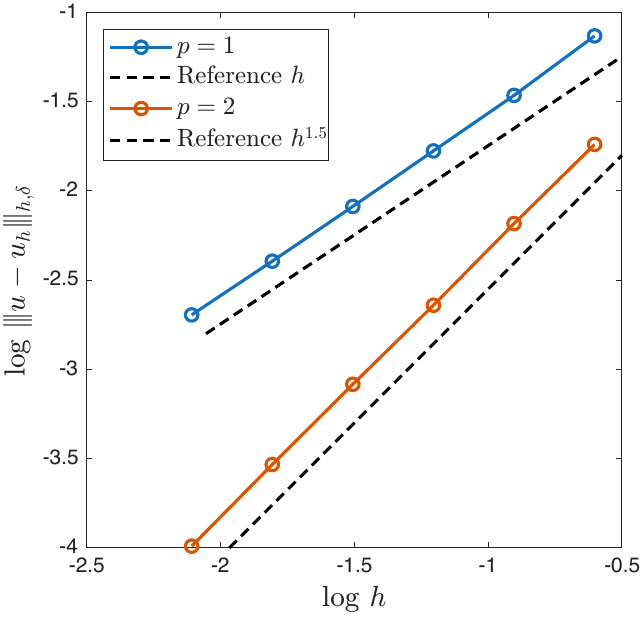}
\caption{Energy norm}
\label{fig:ls_study_E}
\end{subfigure}
\hfill
\begin{subfigure}[b]{0.32\textwidth}
\centering
\includegraphics[width=\textwidth]{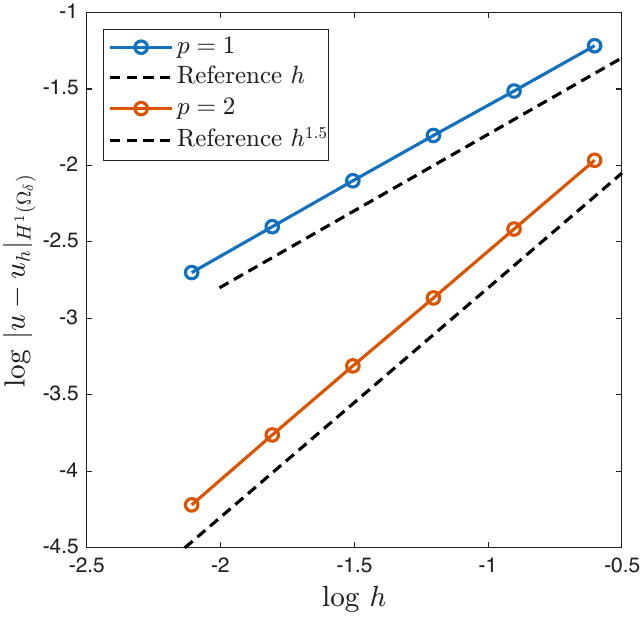}
\caption{$H^1$-seminorm}
\label{fig:ls_study_H1}
\end{subfigure}
\hfill
\begin{subfigure}[b]{0.32\textwidth}
\centering
\includegraphics[width=\textwidth]{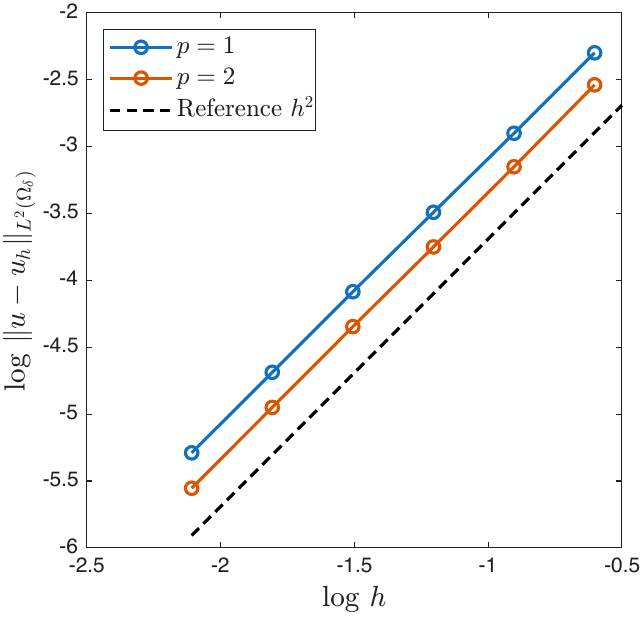}
\caption{$L^2$-norm}
\label{fig:ls_study_L2}
\end{subfigure}

\caption{\emph{Level-set study.} Convergence in the energy norm, $H^1$-seminorm, and $L^2$-norm for polynomial orders $p=1$ and $p=2$ using a piecewise linear level-set representation of the domain. For $p=1$, optimal convergence is observed in all norms. For $p=2$, the observed convergence rates are consistent with the geometry-induced limitations predicted by the analysis: the energy norm and $L^2$-norm exhibit the expected geometry-limited rates, while the $H^1$-seminorm converges suboptimally due to the normal approximation error. The observed $H^1$ rate is, however, slightly higher than predicted, indicating that the worst-case normal error is not fully realized in this setting.}
\label{fig:ls_study}
\end{figure}

\section{Conclusions}
\label{sec:conclusions}

We have analyzed the effect of geometric approximation errors in Nitsche’s method applied to elliptic problems on perturbed domains. The analysis provides a clear separation between discretization errors and geometry-induced errors, and reveals a fundamental distinction between the norms in how these errors affect the solution.

In the energy norm, the boundary location error enters with an $h^{-1/2}$ amplification, while the estimate is insensitive to the normal approximation error. This implies that optimal convergence requires the geometric approximation to satisfy $\delta \sim h^{p+1/2}$. In contrast, the refined $H^1$-seminorm estimate removes this amplification and yields an additive contribution of the boundary location and normal errors, showing that the normal approximation error $\delta_n$ plays a decisive role for convergence in this norm. For the $L^2$-norm, we establish an optimal order estimate in which the geometry contribution appears as a separate additive term, decoupled from the mesh size and insensitive to the normal error.

These results highlight that geometric perturbations affect different norms in qualitatively different ways: the energy norm amplifies boundary location errors while remaining insensitive to normal errors, the $H^1$-seminorm separates boundary location and normal errors, and the $L^2$-norm is comparatively robust with respect to geometric inaccuracies.

The numerical experiments are consistent with the theoretical predictions and illustrate how the different geometric error components influence convergence in practice. In particular, the normal approximation study demonstrates that the $\delta_n$-term in the $H^1$-estimate is sharp, while the level-set study shows that in practical approximations the worst-case scaling of the normal error may not be fully realized, leading in some cases to slightly improved convergence rates.

From a practical perspective, the results provide guidance on the accuracy
needed in geometric representations for unfitted finite element methods.
Boundary location errors need to be carefully controlled to guarantee optimal
convergence in the energy norm, while the boundary normal approximation also
plays a role in guaranteeing optimal convergence in the \(H^1\)-seminorm. The
\(L^2\)-estimate, on the other hand, exhibits a more natural dependence on the
geometry: the boundary location error enters as a separate additive contribution,
and no explicit normal-error term appears.

The analysis is carried out in an abstract CutFEM framework and applies to a wide class of unfitted and approximate-domain methods. Possible directions for future work include the development of error estimates under weaker geometric regularity assumptions, such as geometries controlled only in Sobolev norms through PDE-based descriptions, and the extension of the analysis to polygonal and piecewise smooth geometric approximations.

\bigskip
\paragraph{Acknowledgement.} This research was supported in part by the Swedish Research
Council (2021-04925, 2025-05562), the Swedish
Research Programme Essence, and the Kempe Foundations (JCSMK22-0139).

\bibliographystyle{habbrv}
\footnotesize{
\bibliography{perturbed-nitsche-refs}
}

\bigskip
\bigskip
\noindent
\footnotesize {\bf Authors' addresses:}

\smallskip
\noindent
Mats G. Larson,  \quad \hfill \addressumushort\\
{\tt mats.larson@umu.se}

\smallskip
\noindent
Karl Larsson,  \quad \hfill \addressumushort\\
{\tt karl.larsson@umu.se}

\smallskip
\noindent
Shantiram Mahata,  \quad \hfill Mathematics and Physics, Linnaeus~University, Sweden\\
{\tt shantiram.mahata@lnu.se}

\end{document}